\documentclass[12pt]{amsart}
\usepackage[utf8]{inputenc}
\usepackage[leqno]{amsmath}
\usepackage{amsthm}
\usepackage{amsfonts}
\usepackage{amssymb}
\usepackage{mathrsfs}
\usepackage{url}

\usepackage{bbm}
\usepackage{graphicx}
\usepackage[all,cmtip]{xy}
\usepackage[bookmarks=true]{hyperref}
\usepackage[left=2.5cm,right=2.5cm,top=2cm,bottom=2cm]{geometry}

\theoremstyle{plain}
\newtheorem{theorem}{Theorem}[section]
\newtheorem{proposition}[theorem]{Proposition}

\newtheorem{lemma}[theorem]{Lemma}

\newtheorem{corollary}[theorem]{Corollary}

\theoremstyle{definition}
\newtheorem{definition}[theorem]{Definition}
\newtheorem{remark}[theorem]{Remark}
\newtheorem{example}[theorem]{Example}

\newcommand{\bb}{\mathbb}
\newcommand{\mcal}{\mathcal}

\newcommand{\CC}{{\mathbb{C}}}

\newcommand{\NN}{{\mathbb{N}}}
\newcommand{\PP}{{\mathbb{P}}}

\newcommand{\ZZ}{{\mathbb{Z}}}

\newcommand{\calE}{{\mathcal{E}}}
\newcommand{\calF}{{\mathcal{F}}}
\newcommand{\calG}{{\mathcal{G}}}

\newcommand{\calI}{{\mathcal{I}}}

\newcommand{\calO}{{\mathcal{O}}}
\newcommand{\calS}{{\mathcal{S}}}
\newcommand{\calM}{{\mathcal{M}}}

\newcommand{\bdd}{\mathrm{bdd}}
\newcommand{\cur}{{\sqrt{-1}\Theta}}
\newcommand{\Id}{{\textup{Id}}}

\newcommand{\supp}{{\textup{Supp}\ }}
\newcommand{\loc}{{\textup{loc}}}
\newcommand{\psh}{{\textup{Psh}}}
\newcommand{\rk}{\mathrm{rank}}
\newcommand{\FS}{\mathrm{FS}}

\newcommand{\Grif}{{\textup{Grif}}}
\newcommand{\Nak}{{\textup{Nak}}}
\newcommand{\ndim}{\mathrm{nd}}

 \def \Grif{  { \rm Grif}  }
 \def \Nak{  { \rm Nak}  }
 \def \SNak{  { \rm SNak}  }

 \def \P{  \mathbb{P}  }
 \def \dpa{ \overline{\partial}  }
 \def \pa{  \partial   }

\newcommand{\eps}{\varepsilon}

\renewcommand{\bar}[1]{\overline{#1}}
\renewcommand{\tilde}[1]{\widetilde{#1}}

\newcommand{\ddbar}{\sqrt{-1}\partial\bar{\partial}}

\usepackage{indentfirst}
\usepackage{cancel}
\usepackage{enumerate}
\usepackage{color}
\usepackage{xcolor}
\usepackage{verbatim}
\usepackage{amsmath}
\usepackage{amssymb}
\usepackage{pifont}
\usepackage{ulem}
\usepackage{enumerate}
\allowdisplaybreaks[4]

\title[\textit{L\MakeLowercase{iu}, L\MakeLowercase{iu}, Y\MakeLowercase{ang and} Z\MakeLowercase{hou}}, A L\MakeLowercase{e} P\MakeLowercase{otier-type isomorphism}]{
A Le Potier-type isomorphism twisted with \\ multiplier submodule sheaves}
\author[]{Yaxiong Liu, Zhuo Liu, Hui Yang and Xiangyu Zhou}

\date{}

\begin{document}
\maketitle\thispagestyle{empty}

\begin{abstract}

 In this paper, we obtain a Le Potier-type isomorphism theorem twisted with multiplier submodule sheaves, which relates a holomorphic vector bundle endowed with a strongly Nakano semipositive singular Hermitian metric to the tautological line bundle with the induced metric.
	As applications, we obtain a Koll\'ar-type injectivity theorem, a Nadel-type vanishing theorem, and a singular holomorphic Morse inequality for holomorphic vector bundles and so on.
\end{abstract}
\makeatletter\def\Hy@Warning#1{}\makeatother

\section{Introduction}

		Let $E$ be a holomorphic vector bundle of rank $r$ over a complex manifold $X$.
		Consider the projectivized bundle $\mathbb{P}(E^*)$ of the dual bundle $E^*$ and its tautological line bundle $L_E:=\mathcal{O}_{\PP(E^*)}(1)$.
		Let $\pi$ be the induced projection from $\mathbb{P}(E^*)$ to $X$.
		It is well-known that $\pi_* L_E^m=S^mE$ for $ m\ge1$ where $S^mE$ is the $m$-th symmetric tensor power of $E$. On the one hand, Kobayashi-Ochiai \cite{KO70} obtained isomorphisms
		\begin{align}\label{formula K/X}
		K_{\mathbb{P}(E^*)/X}=(L_E)^{-r}\otimes \pi^*\det E
		\end{align}
  and \[H^q(X,W\otimes S^mE)=H^q(\mathbb{P}(E^*),\pi^*W\otimes L_E^m),\]
  where $W$ is a holomorphic vector bundle over $X$.
On the other hand, via using the spectral sequence, Le Potier \cite{lepotier73} obtained the following isomorphism theorem which establishes a connection between the cohomology of vector bundles and that of line bundles:
		\[H^q(X,\Omega_X^p\otimes E)=H^q(\mathbb{P}(E^*),\Omega_{\mathbb{P}(E^*)}^p\otimes L_E).\]
By the way, Schneider \cite{Sch74} gave a simpler proof of Le Potier's isomorphism by K\"unneth's formula and Bott's vanishing.

		In this paper, we contribute to establishing a generalization of the above  isomorphism theorem twisted with \textit{multiplier ideal $($submodule$)$ sheaves} associated to \textit{singular Hermitian metrics}.
	
		Recall that a singular Hermitian metric $h$ on $E$  is a measurable map from $X$ to the space of non-negative Hermitian forms on the fibers  satisfying $0<\det h<+\infty$ almost everywhere.
		The metric $h$
		is called \textit{Griffiths semi-positive}, denoted  by $(E, h)\ge_\Grif 0 $, if $|u|^2_{h^*}$ is plurisubharmonic for any local holomorphic section $u$ of $E^*$.
		And  $h$ is called  Griffiths positive, denoted  by $(E, h)>_\Grif 0 $, if locally there exists a smooth strictly plurisubharmonic function $\psi$ such that $he^{\psi}$ is  Griffiths semi-positive.
		For any  singular Hermitian metric $h$ on $E$,
		it induces a singular Hermitian metric $h_L$ on $L_E$ by the quotient morphism $\pi^*E\to L_E$ .
		Then it follows from Proposition \ref{pro posi} that the $(L_E,h_L)\ge_\Grif 0$ as soon as $(E,h)\ge_\Grif 0$.
		
		As a tool for characterizing the singularities of Griffiths semi-positive singular Hermitian metrics on holomorphic line bundles, the multiplier ideal  sheaf has played an important role in several complex variables and complex algebraic geometry.  In \cite{Cataldo98}, M. A. de Cataldo  defined
		the multiplier submodule sheaf $\calE(h)$ of $\mathcal{O}(E)$ associated to a singular Hermitian metric $h$  on a holomorphic vector bundle $E$ as follows:
		\begin{equation*}
		\calE(h)_x:=\{u\in E_x:|u|^2_{h} \mathrm{\ is \ integrable\ in\ some\ neighborhood\ of }\ x \}.
		\end{equation*}
		When $E$ is a holomorphic line bundle, $\calE(h)$ is simply $\mathcal{O}(E)\otimes \calI(h)$, where  $\calI(h)$ is the multiplier ideal sheaf associated to $h$.
		
		It's well-known that
		the multiplier ideal sheaf associated to a pseudo-effective line bundle satisfies some basic properties: coherence, torsion-freeness, Nadel vanishing theorem and so on.
		Nevertheless, it remains uncertain whether Griffiths semi-positivity implies the coherence of multiplier submodule sheaves. In more recent developments, Inayama \cite{In22} has showed that in cases where the unbounded locus of $\det h$ is isolated, the Griffiths semi-positivity of $(E,h)$ indeed leads to the coherence of $\calE(h)$. Similarly, Zou \cite{Zou22} has demonstrated that the coherence of $\calE(h)$ is a consequence of the Griffiths semi-positivity of $(E,h)$ when $\det h$ have analytic singularities.
		
To derive further insights into multiplier submodule sheaves, we must delve into notions of positivity that surpass Griffiths positivity.
In  \cite{Wu20}, Wu introduced a new positivity for singular Hermitian metrics on holomorphic vector bundles and showed that this positivity is stronger than Nakano positivity for smooth Hermitian metrics.

\begin{definition}
			\label{SNak}
			We say that a singular Hermitian metric $h$ is  \textit{strongly Nakano semi-positive}, denoted by $(E,h)\ge_\SNak 0 $, if $(E,h)\ge_\Grif 0$ and
			\[(L_E^{r+1}\otimes\pi^*\det E^*,h_L^{r+1}\otimes\pi^*\det h^*)\ge_\Grif 0.\]
			Moreover, we say that a singular Hermitian metric $h$ is  \textit{strongly Nakano positive}, denoted by $(E,h)>_\SNak 0 $, if locally there exists a smooth strictly plurisubharmonic function $\psi$ such that $(E,he^\psi)\ge_\SNak 0$.
\end{definition}		
\begin{remark}
        \begin{enumerate}[$(\rm i)$]
            \item Our definition may appear slightly different from Wu's article, yet it aligns as long as we presume prior Griffiths semi-positivity of $h$.
            \item Let $M$ be a holomorphic line bundle over $X$ endowed with a singular Hermitian metric $h_M$. By abuse of notation, we always  denote by
            \[(E,h)\geq_\SNak \sqrt{-1}\Theta_{(M,h_M)}\] if $(E\otimes M^*, h\otimes h_M^*)\geq_\Grif 0$ and 	
            \[(L_E^{r+1}\otimes\pi^*(\det E^*\otimes M^*),h_L^{r+1}\otimes\pi^*(\det h^*\otimes h_M^*))\ge_\Grif 0.\]
            \item It is easy to check that $(E, h)\ge_\Grif 0$ (resp. $>_\Grif 0$) implies  $(E\otimes \det E, h\otimes \det h)\ge_\SNak 0$ (resp. $>_\SNak 0$) and  strongly Nakano positivity is corresponding to Griffiths positivity when $E$ is of rank one.
            \item	Notice that when a smooth Hermitian metric $h$ is strongly Nakano (semi-)positive, it gives rise to a smooth and (semi-)positive metric $h_F := h_L^{m+r} \otimes \pi^*(\det h^*)$ on $\tilde{F}:= K_{\mathbb{P}(E^*)/X}^* \otimes L_E^m$, where $h_L$ is the metric induced by $h$ on $L_E$. Remarkably, Berndtsson \cite{Bern09} showed that the $L^2$-metric induced by $h_F$ (initially introduced by Narasimhan-Simha \cite{NS68}) on $\pi_*(K_{\mathbb{P}(E^*)/X}\otimes F)=S^mE$ is Nakano (semi-)positive. Additionally, Liu-Sun-Yang \cite[Theorem 7.1]{LSY13} showed that the $L^2$-metric is just a constant multiple of $S^mh$. Thus strong Nakano positivity implies Nakano positivity for smooth Hermitian metrics.
            \item We will show that strong Nakano positivity implies Nakano positivity too for singular Hermitian metrics (See Proposition \ref{thm strong implies NP}).		
         \end{enumerate}
\end{remark}

Wu proposed \cite[Problem 5.1]{Wu20} on the potential equivalence between Nakano positivity and strongly Nakano positivity. We consider the decomposable vector bundle of rank two over Riemann surface, whose dual projectivized bundle is the so-called \textit{Hirzebruch-like ruled surfaces} and prove that it is Nakano (semi-)positive, but not \textit{strongly} Nakano (semi-)positive (see Example \ref{nak not snak}). This gives a negative answer to Wu's problem.

		With these notions, we obtain a Le Potier-type isomorphism theorem for holomorphic vector bundles with strongly Nakano semi-positive singular Hermitian metrics.
		
\begin{theorem}[Main Theorem]\label{Isom thm}
	Let $E$ be a holomorphic vector bundle of rank $r$ over a complex manifold $X$.
	Assume that $(E,h)\ge_\SNak 0$.
\begin{enumerate}[$(\rm i)$]\setlength{\itemsep}{2ex}
  \item If $W$ is a holomorphic vector bundle over $X$, then for any $q\ge 0$ and $ m\ge1$, we have
\begin{equation*}
  H^q(X, W\otimes {S}^m\calE(S^mh))
		\simeq
		H^q(\P(E^*),\pi^*W\otimes L_E^{m}\otimes\calI(h_L^{m+r}\otimes\pi^*\det h^*) ),
\end{equation*}
		  \item and for any $p,q\ge 0$, we have
\begin{equation*}
  H^q(X, \Omega_X^p\otimes \calE(h))
		\simeq
		H^q(\P(E^*),\Omega_{\PP(E^*)}^p\otimes L_E^{}\otimes\calI(h_L^{1+r}\otimes\pi^*\det h^*) ).
\end{equation*}
\end{enumerate}
\end{theorem}

\begin{remark}
  \begin{enumerate}[$(\rm i)$]
    \item
	By the formula $(\ref{formula K/X})$, $L_E^{m}\otimes\calI(h_L^{m+r}\otimes\pi^*\det h^*)$ is actually
	\[K_{\mathbb{P}(E^*)/X}\otimes L_E^{m+r}\otimes\pi^*\det E^*\otimes\calI(h_L^{m+r}\otimes\pi^*\det h^*).\]
    \item
    It is well-known that  the multiplier ideal sheaf itself does Not have the factorial property under proper modifications  but obtains it when twisted with the canonical line bundle. This explains our focus on the isomorphism
    \[\pi_*(K_{\mathbb{P}(E^*)}\otimes (L_E)^{m+r}\otimes \pi^*\det E^*)=K_X\otimes S^mE\]
rather than $\pi_*(L_E)^m=S^mE$.
    \item We actually prove the main theorem when there exists locally a smooth function $\psi$ such that $he^{-\psi}$ is strongly Nakano semi-positive. Especially,  if $h$ is smooth, then muiltiplier submodule sheaves become trivial and the above theorem becomes Le potier's isomorphism and Kobayashi-Ochiai's isomorphism.
  \end{enumerate}
\end{remark}

Our proof is different from Le potier's and Schneider's. Due to Liu-Sun-Yang's formula \cite{LSY13}, we can show that
\begin{equation}\label{11111}
  \pi_*(K_{\mathbb{P}(E^*)/X}\otimes L_E^{m+r}\otimes\pi^*\det E^*\otimes\calI(h_L^{m+r}\otimes\pi^*\det h^*))\simeq\calS^m\calE(S^mh),
\end{equation}
for Griffiths semi-positive singular Hermitian metrics $h$.
Therefore, Theorem \ref{Isom thm}-(i) can be deduced from Leray's isomorphism theorem if one can show the vanishing of higher direct image sheaves
\begin{equation}\label{22222}
   R^k\pi_*(K_{\mathbb{P}(E^*)/X}\otimes L_E^{m+r}\otimes\pi^*\det E^*\otimes\calI(h_L^{m+r}\otimes\pi^*\det h^*)).
\end{equation}
When $h$ is smooth, (\ref{22222}) can be obtained by K\"unneth's formula and Bott's vanishing.
However, the approach fails when twisted with mutiplier submodule sheaves. Noticing that $h_L^{m+r}\otimes\pi^*\det h^*$ is positive along the fiber $\pi^{-1}(x)$ such that $\det h(x)$ is finite, the strongly Nakano semi-positivity of $h$ assures that the higher direct images vanishes (see Lemma \ref{Direct image vanishing}).
In fact, we only need to assume that there is locally a smooth function $\psi$ such that $he^{-\psi}$ is Griffiths semi-positive and strongly Nakano semi-positive respectively.
In particular, we give a new proof of Le potier's isomorphism.

    It is noteworthy that (\ref{11111}) offers insight into exploring the coherence of $\calE(h)$ associated to a Griffiths semi-positive vector bundle $(E,h)$ through an investigation of the coherence of $\mathcal{I}(h_L^{1+r}\otimes\pi^*\det h^*)$, owing to Grauert's direct image theorem.

\begin{proposition}[Proposition \ref{Coherence}]
    Let $(E,h)\ge_\Grif 0$.
    \begin{enumerate}[$(\rm i)$]
        \item  If $\calI(h_L^{1+r}\otimes\pi^*\det h^*)$ is coherent, then $\calE(h)$ is coherent.
        \item  If $\det h$ has a discrete unbounded locus or analytic singularities, then  $\calI(h_L^{1+r}\otimes\pi^*\det h^*)$ is coherent.
    \end{enumerate}
\end{proposition}

	Similarly, we can also  obtain the following isomorphism theorem.

\begin{theorem}
\label{Isom thm'}
	Let $E$ be a holomorphic vector bundle of rank $r$ over a complex manifold $X$.
	Assume that $(E,h)\ge_\Grif 0$,
\begin{enumerate}[$(\rm i)$]\setlength{\itemsep}{2ex}
  \item then for $q\ge 0, m\ge1$, we have
  \[
		H^q(X, K_X\otimes {S}^m\calE\otimes\det \calE(S^mh\otimes\det h))
		\simeq
		H^q(\P(E^*),K_{\P(E^*)}\otimes L_E^{m+r}\otimes\calI(h_L^{m+r}) ),
\]
  \item and for $p,q\ge 0$, we have
\[
H^q(X, \Omega_X^p\otimes \calE\otimes\det\calE(h\otimes\det h))
		\simeq
		H^q(\P(E^*),\Omega_{\P(E^*)}^p\otimes L_E^{}\otimes\pi^*\det E \otimes\calI(h_L^{1+r} )).\]
\end{enumerate}
\end{theorem}

	Theorem \ref{Isom thm} and Theorem \ref{Isom thm'} builds a bridge between the multiplier submodule sheaves and the multiplier ideal sheaves on its dual projectivized bundles. This connection offers a pathway to explore the properties of multiplier submodule sheaves by leveraging the well-developed theory of multiplier ideal sheaves, which includes Koll\'ar-type injective theorems (\cite{Kol86,Fuj12,FM21,ZZ}), Nadel-type vanishing theorems (\cite{Nadel,C14,MZ19}), and holomorphic Morse inequalities (\cite{D12b,MM07,DX21}).

Firstly, Zhou-Zhu \cite{ZZ} derived an injective theorem for pseudo-effective line bundles on holomorphically convex manifolds (see Theorem \ref{ZZ inj}). Consequently, using Theorem  \ref{Isom thm}, we can establish a Koll\'ar-type injective theorem for strongly Nakano positive vector bundles.
\begin{corollary}[Theorem \ref{inj for vb}]\label{cor inj}
	Let $(X,\omega)$ be a holomorphically convex K\"ahler manifold and $(E,h)\geq_\SNak0$.
	Assume that
	\begin{equation*}
	(E,h)\geq_\SNak b\sqrt{-1}\Theta_{(M,h_M)}
	\end{equation*}
	for some $0<b<+\infty$,
	and
	\begin{equation*}
	\sqrt{-1}\Theta_{(M,h_M)}\ge-C\omega
	\end{equation*}
	for some constant $C$.
	Then for any non-zero $s\in H^0(X,M)$ satisfying
	\begin{equation*}
	\sup_\Omega |s|_{h_M}<+\infty
	\end{equation*}
	for every $\Omega\Subset X$, we obtain that the map
	\begin{equation*}
	H^q(X,K_X\otimes S^m\calE(S^mh))\xrightarrow{\otimes s}
	H^q(X,K_X\otimes S^m\calE\otimes \calM(S^mh\otimes h_M))
	\end{equation*}
	is injective for any $q\ge 0$.
\end{corollary}

	Secondly, Cao \cite{C14} established a Kawamata-Viehweg-Nadel-type vanishing theorem for a pseudo-effective line bundle $(M,h_M)$ twisted with the upper continuous variant $\calI_+(h_M)$ of the multiplier ideal sheaf.
	Additionally, Guan-Zhou \cite{GZ15} proved Demailly's strong openness conjecture: $\calI_+(h_M)=\calI(h_M)$.
	Therefore, by applying Theorem \ref{Isom thm}, we can deduce a Nadel-type vanishing theorem for strongly Nakano semi-positive vector bundles on compact K\"ahler manifolds.

\begin{corollary}[Corollary \ref{Cao type vanishing}]
	Let  $(E, h)\ge_\SNak 0$  be a holomorphic vector bundle of rank $r$ over a compact K\"ahler manifold $X$ of dimension $n$.
	then for any $m\ge 1$, we have
	\begin{equation*}
		H^q(X,K_X\otimes S^m\calE(S^mh))=0
	\end{equation*}
	for any $\ q\ge n+r-{\rm nd}(L_E^{m+r}\otimes\pi^*\det E^*,h_L^{m+r}\otimes\pi^*\det h^*)$.
\end{corollary}

Notice that if  $(E, h)>_\SNak 0$ over a compact (possibly non-K\"ahler) manifold $X$, then $X$ admits a K\"ahler modification and ${\rm nd}(L_E^{m+r}\otimes\pi^*\det E^*,h_L^{m+r}\otimes\pi^*\det h^*)=n+r-1$.
In addition, Corollary \ref{cor inj} conforms that cohomology groups remain invariant under K\"ahler modifications. Hence, we can conclude that

\begin{corollary}[Corollary \ref{coro vanisging for SNak>0}]
 Let  $(E, h)>_\SNak 0$  be a holomorphic vector bundle over a compact $($possibly non-K\"ahler$)$ manifold $X$, then for any $m\ge1$, we have
	\begin{equation*}
		H^q(X,K_X\otimes S^m\calE(S^mh))=0
	\end{equation*}
	for  $ q\ge  1$.
\end{corollary}

If $(E,h)$ is merely Griffiths (semi-)positive, applying Theorem \ref{Isom thm'}, we can also derive the corresponding injectivity and vanishing theorems for $\calS^m\calE\otimes\det \calE(S^mh\otimes\det h)$ by a similar argument.

Thirdly,  Bonavero \cite{Bon98} obtained Demailly-type holomorphic Morse inequalities twisted with multiplier ideal sheaves associated to singular metrics with analytic singularities.
Consequently, employing Theorem  \ref{Isom thm'}, we can also conclude singular holomorphic Morse inequalities for holomorphic vector bundles.

\begin{corollary}[Theorem \ref{Singular Morse for v.b.}]
	Let $(X,\omega)$ be a compact complex manifold of dimension $n$ and $(E,h)$ a holomorphic vector bundle endowed with a singular Hermitian metric. Assume that there is a smooth function $\phi$ such that $(E,he^{-\phi})\ge_\Grif 0$ and  the induced metric $h_L$ on $L_E$ over $\PP(E^*)$ has analytic singularities.
	Let $V$ be a holomorphic vector bundle on $X$.
	Then for $0\leq q\leq n$, we have
	\begin{align*}
	&h^q(X, V\otimes S^m\calE\otimes\det \calE(S^mh\otimes\det h))  \\
	\leq
	& \rk(V) \frac{(m+r)^{n+r-1}}{(n+r-1)!} \int_{\PP(E^*)(q)} (-1)^q(\cur_{(L_E,h_L)})^{n+r-1} +o(m^{n+r-1}),
	\end{align*}
	and
	\begin{align*}
	&\sum_{j=0}^q(-1)^{q-j} h^j(X, V\otimes S^m\calE\otimes\det \calE(S^mh\otimes\det h)) \nonumber \\
		\leq
	& \rk(V) \frac{(m+r)^{n+r-1}}{(n+r-1)!} \int_{\PP(E^*)(\leq q)} (-1)^q(\cur_{(L_E,h_L)})^{n+r-1} +o(m^{n+r-1})
\end{align*}
	as $m\to +\infty$ $($see \textup{Section \ref{Singular holomorphic Morse inequalities for vector bundles}} for the definitions of $\PP(E^*)(q)$ and $\PP(E^*)(\leq q)$$)$.
\end{corollary}

	Finally, based on Theorem \ref{Isom thm'} and Darvas-Xia's \cite{DX21} asymptotic formula (Theorem \ref{DX asymptotic}) for the volumes of pseudo-effective line bundles, we obtain the following asymptotic formula for Griffiths positive vector bundles.

\begin{corollary}[Theorem \ref{Darvas-Xia type}]
	Let $X$ be a compact K\"ahler  manifold and $(E,h)\ge_\Grif 0$.
	Then we have
	\begin{align*}
	&\ \lim_{m\rightarrow\infty}\frac{1}{m^{n+r-1}}
	h^0(X,S^m\calE\otimes\det \calE(S^mh\otimes\det h))\\
	=\ &\ \frac{1}{(n+r-1)!}\int_{\PP(E^*)} (\cur_{(L,P[h_L]_\calI)})^{n+r-1}.
	\end{align*}
	where $P[h_L]_\calI\in\psh(X,L)$ is the so-called $\calI$-model envelope of positive metric $h_L$ $($see \textup{Section \ref{Singular holomorphic Morse inequalities for vector bundles}}$)$
	and $(\cur_{(L,P[h_L]_\calI)})^{n+r-1}$ is the Monge-Amp\`ere measure of $P[h_L]_\calI$ $($see \textup{Section \ref{Nadel-Cao type vanishing theorem}}$)$.
\end{corollary}

In particular, we obtain an asymptotic inequality for  vector bundles with Griffiths positive singular Hermitian metrics on compact manifolds, which can be viewed as a generalization of the classical result of big line bundles.

\begin{theorem}[Theorem \ref{thm nonvan}]
	Let  $(E, h)\ge_\Grif 0$  be a holomorphic vector bundle of rank $r$ over a compact $($possibly non-K\"ahler$)$ manifold $X$ of dimension $n$. Then we have	
 \begin{align*}
	\lim_{m\rightarrow\infty}\frac{1}{m^{n+r-1}}
	h^0(X,S^m\calE(S^mh))
	\geq\frac{1}{(n+r-1)!}\int_{\PP(E^*)} \left(\cur_{(L_E,h_L)}\right)^{n+r-1}.
\end{align*}
\end{theorem}

This article is organized as follows.
\begin{itemize}
	\item In Section \ref{Preliminaries},we revisit fundamental concepts regarding the positivity of singular Hermitian metrics and $L^2$-metrics on the direct image sheaves.
	\item In Section \ref{Proof of Main Theorems}, we present a proof of our main Theorem \ref{Isom thm}.
	\item In Section \ref{Applications}, we explore various applications from Theorem \ref{Isom thm} via connecting with results of pseudo-effective line bundles.
\end{itemize}

\section{Preliminaries}
\label{Preliminaries}

The primary aim of this preliminary section is to provide a comprehensive overview of the fundamental theory concerning the positivity of metrics on holomorphic vector bundles and Berndtsson's $L^2$-metrics. Additionally, we revisit classical results in complex analytic geometry, crucial for our proof.\subsection{Notation and convention}

\begin{itemize}
	\item Throughout the paper, unless explicitly specified otherwise, all metrics are considered singular Hermitian metrics.
    \item We fix $X$ as a complex manifold of dimension $n$ and $E$ as a holomorphic vector bundle of rank $r$ on $X$.
    When we say $(X,\omega)$ is a complex (resp. Hermitian, K\"ahler) manifold, it means that $\omega $ is a smooth (resp. Hermitian, K\"ahler) metric on $X$.

    \item For the convenience of writing, we do not differentiate between the Hermitian metrics of line bundles and its (global) metric weights.
    \item To maintain consistency in notation, we denote metrics on holomorphic vector bundles as $h$, metrics of holomorphic line bundles as $\varphi$, and the $L^2$-metric of $E$ induced by the metric $\varphi$ of $L_E$ as $h_\varphi$.
    \item We consistently use $L_E$ to represent the tautological line bundle $\mathcal{O}_E(1)$ of $E$ over $\mathbb{P}(E^*)$.
    \item The metric of $L_E$ induced by the metric $h$ is denoted as $h_L$. Additionally, in cases where no confusion arises, we use $\varphi_h$ for the metric weight of $h_L$.
\end{itemize}

\subsection{$L^2$-metrics on the direct image sheaves}

Recall some notations from \cite[\S 2.1]{Bern15}.

Let $F$ be a holomorphic line bundle over a complex manifold $Y$ of dimension $\ell$ and $\{U_j\}$ be an open covering of the manifold such that $F$ is  trivial on each $U_j$.
A section $s$ of $F$ consists of complex-valued functions $s_j$ on $U_j$ satisfying $s_j = g_{jk}s_k$, where $g_{jk}$ represents the transition functions of the bundle.
A metric (weight) is a set of real-valued functions $\psi^j$ on $U_j$  such that the globally well-defined expression
\begin{equation*}
	|s|^2_\psi
	:=|s|^2e^{-\psi}
	:=|s_j|^2 e^{-\psi^j}
\end{equation*}
holds.
We denote $\psi$ to represent the collection $\psi^j$.

A metric $\psi$ on $F$ induces an $L^2$-metric on the adjoint bundle $K_Y\otimes F$.
A section $\xi$ of $K_Y\otimes F$ can be locally expressed as
\begin{equation*}
	\xi=dz\otimes s,
\end{equation*}
where $dz=dz^1\wedge\cdots\wedge dz^\ell$ for some local coordinate of $Y$ and $s$ is a local section of $F$.
We define
\begin{equation*}
	|\xi|^2e^{-\psi}
	:=c_n dz\wedge d\bar z |s|^2_\psi,
\end{equation*}
which represents a (global) volume form on $Y$.
The $L^2$-norm of $\xi$ is given by
\begin{equation*}
	\|\xi\|^2
	:=\int_Y|\xi|^2e^{-\psi}.
\end{equation*}
It is important to note that the $L^2$-norm only depends on the metric $\psi$ on $F$, but does not involve any choice of metrics on the manifold $Y$.

Consider a holomorphic proper fibration $\pi:Y\rightarrow X$ between complex manifolds $X$ and $Y$, and a holomorphic line bundle $F$ on $Y$ with a Griffiths semi-positive smooth metric $\varphi$.
Assume that
 $\pi_{*}(K_{Y/X}\otimes F)$ is locally free and thus $(\pi_{*}(K_{Y/X}\otimes F))_x=H^0(Y_x,K_{Y_x}\otimes F|_{Y_x})$. Berndtsson introduced an $L^2$-metric $h_\varphi$ on $\pi_{*}(K_{Y/X}\otimes F)$ by defining
\begin{equation*}
	|u|^2_{h_\varphi}
	:=\int_{Y_x}|u|^2e^{-\varphi},
\end{equation*}
where $u\in (\pi_{*}(K_{Y/X}\otimes F))_x$ and $Y_x=\pi^{-1}(x)$.
\begin{theorem}[\cite{Bern09}]\label{thm ber09}
  Assume that $h_F$ is smooth positive, then the  $L^2$-metric on $\pi_{*}(K_{Y/X}\otimes F)$ is Nakano positive.
\end{theorem}

Now, let us assume $h$ to be a Griffiths positive singular Hermitian metric on $E$.
Then $h$ induces a singular metric $h_L$ on $L_E$.
Consider a coordinate $(z_1,\ldots,z_n)$ around $x\in X$ and any holomorphic frame $e=\{e_1,\ldots, e_r\}$ of $E$.
The dual frame is $e^*=\{e^*_1,\ldots, e^*_r\}$ of $E^*$.
Let $W^1,\ldots, W^r$ be the coordinate of $E^*$ with respect to $e^*$.
The homogeneous coordinate is represented as $[W^1:\cdots: W^r]$.
On $\{W^1\not=0\}$,  $w^j:=W^j/W^1,j=2,\ldots,n$ is a local coordinate along fibers.

For any $a\in \PP(E^*_x)\backslash\{W_1=0\}$,  $(z_1,\ldots,z_n,w^2,\ldots,w^r)$ defines a local coordinate on $\PP(E^*)$ near $a$.
We can define a local non-vanishing holomorphic section of $\mcal{O}_{E}(-1)=\mcal{O}_{E}(1)^*\subset \pi^*E^*$ over ${\bb P(E^*)}$ as
\[ s:(z,w)\mapsto \sum_{1\le \alpha\le r} w^\alpha e^*_\alpha
=e^*_1+\sum_{2\le \alpha\le r} w^\alpha e^*_\alpha.\]
The metric $h^*_L$ on $\mathcal{O}_{E}(-1)$ induced by $h$ is given by
\[ |s|^2_{h^*_L}=\sum_{1\le \alpha,\beta\le r} w^\alpha \overline{w^\beta} \left<e^*_\alpha,e^*_\beta\right>_{h^*}
=\sum_{1\le \alpha,\beta\le r} w^\alpha \overline{w^\beta} h^*_{\alpha\beta}.\]
Hence the metric $h_L$ on $L_E$ is obtained as
\begin{equation}
\label{induced metric h_L}
h_L=\frac{1}{\sum_{1\le \alpha,\beta\le r} w^\alpha \overline{w^\beta} h^*_{\alpha\beta}}.
\end{equation}

Additionally, let $h$ be smooth and $\{e_\alpha\}$ be a normal frame of $E$ at $x\in X$. Then
the curvature of $h$ at $x$ is given by
\begin{equation*}
	\Theta_{(E,h)}=c_{ij\alpha\beta}dz^i\wedge d\bar z^j\otimes e^*_\alpha\otimes e_\beta ,
\end{equation*}
satisfying $\bar c_{ij\alpha\beta}=c_{ji\beta\alpha}$.

At any point $b\in\PP(E^*_x)$ represented by a vector $\sum b_\alpha e^*_\alpha$ of norm $1$, then the curvature of $h_L$ is derived as
\begin{flalign}
	\Theta_{(L_E,h_L)}(b)
	=&\sum c_{ij\beta\alpha }b_\alpha\bar{b_\beta} dz^i\wedge d\bar z^j +\sum dw^\lambda\wedge d\bar w^\lambda \nonumber \\
	=&\langle -\Theta(h^*)b,b\rangle
	+\sum dw^\lambda\wedge d\bar w^\lambda.
	\label{curvature of LE normal}
\end{flalign}

\begin{theorem}[\cite{LSY13}]\label{thm lsy}
 Assume that  $h$ is a Griffiths positive smooth Hermitian metric on $E$. Take $Y=\PP(E^*)$ and
 \[(F,h_F)=(L_E^{m+r}\pi^*\det E^*,h_L^{m+r}\otimes\pi^*\det E^*),\]
 Then the $L^2$-metric induced by $h_F$ on
 \[\pi_{*}(K_{Y/X}\otimes F)=S^mE\] is exactly a constant multiple of $S^mh$.
\end{theorem}

We will establish the validity of the above theorem for a Griffiths semi-positive singular Hermitian metric $h$ (refer to Section \ref{Proof of Theorem ref isom thm}).

Finally, let us briefly introduce the results of the $L^2$ metric on direct image sheaves in the singular setting. For a projective fibration between compact complex manifold, Berndtsson-P\u{a}un in \cite{BP} showed  the Griffiths semi-positivity of the $L^2$-metric on the direct image bundle.
Through the optimal $L^2$ extension theorem presented by B{\l}ocki in \cite{Blocki13} and Guan-Zhou in \cite{GZ}, Hacon-Popa-Schnell in \cite{HPS18} expanded Berndtsson-P\u{a}un's result to encompass the direct image sheaf .
Zhou-Zhu in \cite{ZZ math ann} extended their work to K\"ahler case.
Recently,
Zou \cite{Zou21} and  Watanabe \cite{watanabe} improved the positivity of such canonical singular metrics.
For more insights into these results, we refer the readers to \cite[etc]{LY14, T16, PT18, DNWZ20}.

\subsection{Some results used in the proof}
In this subsection, we review several classical results, which are used in the proof of main theorems and applications.

\begin{proposition}[\cite{BP, Rau15, PT18}]
\label{prop lowerbound}
	If $(E,h)\ge_\Grif 0$, then locally there exists a family of  Griffiths positive smooth Hermitian metrics $h_v$ such that $h_v$ increasingly converges to $h$.
\end{proposition}

	In fact, the approximating sequence is obtained through the
	technique of convolution with an approximate identity.
	Hence one immediately yields the following proposition.

\begin{proposition}[\cite{PT18}]
  \label{pro posi}
	Let $(E, h)\ge_\Grif 0$, then
\begin{enumerate}[$(\rm i)$]\setlength{\itemsep}{2ex}
	\item
	$(Q, h_Q)\ge_\Grif 0$, where $(Q, h_Q)$ is a quotient of $(E, h)$,
	\item
	$(S^mE, S^m h)\ge_\Grif 0$ for $m\ge1$,
	\item
	$(\Lambda^m E, \Lambda^m h)\ge_\Grif 0$,
	especially	$(\det E, \det h)\ge_\Grif 0$ for $m\ge1$,
	\item
	$(\pi^*E, \pi^*h)\ge_\Grif 0$, where $\pi:Y\to X$ is a holomorphic morphism.	
\end{enumerate}	
\end{proposition}

\begin{theorem}[R\"uckert's Nullstellensatz, {\cite[Section 3.1.2]{GR84}}]
\label{Ruckert Nullstellensatz}
	Let $\calF$ be a coherent analytic sheaf on a complex space $X$, and let $f\in \calO_X(X)$ be a holomorphic function vanishing on $\supp\calF$.
	Then, for each point $x\in X$, there exist an open neighborhood $U$ of $x$ in $X$ and a positive integer $t$ such that $f^t\calF_U=0$.
\end{theorem}

\begin{theorem}[Remmert's proper mapping theorem, {\cite[Section 10.6.1]{GR84}}]
\label{Remmert's proper mapping theorem}

	For any proper holomorphic map $f: X\to Y$ between complex spaces $X$, $Y$,
	the image set $f(X)$ is an analytic set in $Y$.
\end{theorem}
\begin{theorem}[Grauert's upper semicontinuity theorem, {\cite[Section 10.5.5]{GR84}}]
\label{Grauert's upper semicontinuity}
Let $X$, $Y$ be complex manifolds.  Assume $f:X\to Y$ is a proper submersion. Let $V$ be a holomorphic vector bundle on $X$.
     Then the set $A_{i,d}:=\{y\in Y~|~\dim_\CC H^i(X_y,V_y)\ge d \}$ is an analytic subset of $Y$, where $X_y:=f^{-1}(y)$ and $V_y:=V|_{X_y}$.
	
\end{theorem}

\begin{theorem}[Leray's isomorphism theorem, {\cite[Section 3.5, Theorem 5.5]{K}}]
		\label{Leray thm}
Let $X$ and $Y$ be topological spaces and $\pi:X\rightarrow Y$ be a proper map.
		Given a sheaf $\mathcal{F}$ of abelian groups over $X$, ${R}^q\pi_*\mathcal{F}$ is the higher direct image sheaf over $Y$, defined as the sheafification of presheaf
\begin{equation*}
		U\subset Y\mapsto H^q(\pi^{-1}(U),\mathcal{F}).
		\end{equation*}
		If for some fixed $p$,
		\begin{equation*}
		{R}^q\pi_*\mathcal{F}=0, \quad\mathrm{for\ all}\ q\neq p.
		\end{equation*}
Then for any $i$, there is a natural isomorphism
		\begin{equation*}
		H^i(X,\mathcal{F})\simeq
		H^{i-p}(Y, {R}^p\pi_*\mathcal{F}).
		\end{equation*}
\end{theorem}

\begin{theorem}[{Bott's vanishing theorem, \cite[Section 3.4, Theorem 4.10]{K}}]
\label{Bott vanishing}
	For $p,q\geq0$ and $k\in\ZZ$, then
	\begin{equation*}
		H^q(\PP^n,\Omega_{\PP^n}^p\otimes\calO_{\PP^n}(k))=0
	\end{equation*}
	with the following exceptions:
	\begin{enumerate}[$(\rm i)$]
		\item $p=q$ and $k=0$,
		\item $q=0$ and $k>p$,
		\item $q=n$ and $k<p-n$.
	\end{enumerate}
\end{theorem}

\vskip0.7cm

Now let us recall the definition of numerical dimension of singular metrics and the Kawamata-Viehweg-Nadel type vanishing theorem in \cite{C14}.

Let $(M,h_M)\ge_\Grif 0$ be a holomorphic line bundle over a compact Hermitian manifold $(X,\omega)$.
A sequence $\{h_j\}$ is called a \textit{quasi-equisingular approximation} of $h$ if it satisfies the following conditions:
\begin{enumerate}[(i)]
	\item $-\log h_j$ converges to $-\log h$ in $L^1$ topology and
 \[\cur_{(M,h_j)}\ge \cur_{(M,h_M)}-\delta_j\omega\]
	  for some positive number $\delta_j$ with $\delta_j\rightarrow0$ as $j\rightarrow+\infty$;
	\item all $-\log h_j$ have analytic singularities and $
	-\log h_{j+1}\le -\log h_j+O(1);
	$
	\item for any $\delta>0$ and $m\in\NN$, there exists $j_0=j(\delta,m)\in\NN$ such that
	$
	\calI(h_j^{m(1+\delta)})\subseteq \calI(h_M^m)
	$
	for any $j\geq j_0$.
\end{enumerate}

It is well-known that the existence of the quasi-equisingular sequence is ensured by Demailly's approximation theorem.

For any closed smooth semi-positive $(n-q,n-q)$ form $u$, one defines
\begin{equation}
\label{Cao's definition}
\int_X\left (\cur_{(M,h_M)}\right)^q\wedge u
:=\limsup\limits_{k\rightarrow\infty}\int_X \left(\cur_{(M,h_j)}\right)_{ac}^q\wedge u,
\end{equation}
where $(\cur_{(M,h_j)})_{ac}$ denotes the absolutely continuous part of the current $\cur_{(M,h_j)}$.

In \cite{C14},  when $(X,\omega)$ is a compact K\"ahler manifold, Cao proved that the limsup is independent of the choice of quasi-equisingular approximation and is a limit.

\begin{definition}[\cite{C14}]
	\label{numerical dim} The numerical dimension of $(M,h_M)$ is defined as
	\begin{equation*}
	{\rm nd}(M,h_M)
	:=\max\{q\in \NN~|~  \left(\cur_{(M,h_M)}\right)^q\neq0\}.
\end{equation*}
\end{definition}
By the vanishing theorem obtained by Cao in \cite{C14}  together with the strong openness property of multiplier ideal sheaves obtained by  Guan-Zhou in \cite{GZ15}, one can obtain
\begin{theorem}[\cite{C14,GZ15}]
	\label{K-V-N-C}
	Let $(M,h_M)\ge_\Grif 0$ be a holomorphic line bundle on a compact K\"ahler manifold $(X,\omega)$.
	Then
	\begin{equation*}
	H^q(X,K_X\otimes M\otimes\calI(h_M))=0,\quad\text{for all } q\ge n-{\rm nd}(M,h_M)+1.
	\end{equation*}
\end{theorem}

\begin{definition}
	A holomorphic morphism $\pi:Y\rightarrow X$ is called a \textit{K\"ahler modification} of $X$ if $\pi$ is proper, and  $Y$ is a K\"ahler manifold and  there exists a nowhere dense closed analytic subset $B\subseteq Y$ such that the restriction $\pi|_{Y\backslash B}:Y\backslash B\rightarrow X\backslash \pi(B)$ is an isomorphism.
\end{definition}

\begin{lemma}\label{lem kahler mod}
	Let $(M,h_M)\ge_\Grif 0$ be a holomorphic line bundle on a compact complex manifold $X$ with analytic singularties.
	Suppose
	\begin{equation*}
		\int_X (\cur_{(M,h_M)})^{\dim X}>0
	\end{equation*}
	in the sense of $($\textup{\ref{Cao's definition}}$)$.
	Then $X$ admits a K\"ahler modification.
\end{lemma}
\begin{proof}
	By (\ref{Cao's definition}) and Bonavero's singular Morse inequality (see \cite[Theorem 2.2.30]{MM07}), the line bundle $M$ is big.
	Thus, $X$ is a Moishezen manifold, i.e.,
	there exists a projective algebraic variety $Y$ with $\dim Y=\dim X$ and a bimeromorphic map $\mu:X \dashrightarrow Y$.
	By Hironaka's desingularization theorem, there exists a K\"ahler modification of $X$.
\end{proof}

Additionally, if one wishes to weaken the condition of the underlying manifolds, then we have Meng-Zhou's vanishing Theorem for holomorphically convex K\"ahler manifolds.
\begin{theorem}[\cite{MZ19}]
\label{MZ}
	Let $X$ be a holomorphically convex K\"ahler manifold and $(L, h_L)\geq_\Grif0$ be a holomorphic vector bundle.
	 Suppose $h_L$ is smooth outside an analytic subset $Z$, satisfying $(\sqrt{-1}\Theta_{(L,h_L)})^\ell(x_0)\neq0$ for some $x_0\in X\backslash Z$.
	 Then, for any $q\geq n-\ell+1$, we have
	 \begin{equation*}
	 	H^q(X,K_X\otimes  L\otimes\calI(h_L))=0.
	 \end{equation*}
\end{theorem}

The following lemma explains the factorial property of submodule sheaves under proper holomorphic modifications.

\begin{lemma}[\cite{SY23}]\label{sheaf under modification}

	Let $\mu:\tilde X\rightarrow X$ be a proper holomorphic modification
	and $(E, h)\ge_\Grif 0$ be a holomorphic vector bundle on a complex manifold $X$.
	Then
	\begin{equation*}
		\mu_*(K_{\tilde X}\otimes \mu^*\calE(\mu^*h))
		=K_X\otimes \calE( h),
	\end{equation*}
where $\mu^*\calE(\mu^* h)$ is the multiplier submodule sheaf associated to $(\mu^*E,\mu^*h)$.
\end{lemma}
\begin{proof}
	By definition, there exists an analytic subset $A\subsetneq X$ such that the map
	\begin{equation*}
		\mu:\tilde X\backslash\mu^{-1}(A) \rightarrow X\backslash A
	\end{equation*}
	is a holomorphic bijection.
	
	For any open subset $U\subseteq X$, $f\in H^0(U,K_X\otimes \calE(h))$ is a $E$-valued holomorphic $(n,0)$-form on $U$ with $|f|_h^2\in L_{\loc}^1(U)$.
	Here, $\mu^*f$ denotes the pullback of the restriction on $U\backslash A$ of $f$.

	Then, one has
	\begin{equation*}
		\mu^*f \in H^0(\mu^{-1}(U)\backslash\mu^{-1}(A), K_{\tilde X}\otimes\mu^*E)
	\end{equation*}
	and
	\begin{equation*}
		\int_{\mu^{-1}(V)\backslash\mu^{-1}(A)} |\mu^*f|^2_{\mu^*h}
		=\int_{V\backslash A}|f|^2_h <+\infty
	\end{equation*}
	for any $V\Subset U$.
	Since $\mu^*h$ has a positive lower bound locally,  $\mu^*f$ can be  extended holomorphically across the analytic subset $\mu^{-1}(A)$, i.e.,
	\begin{equation*}
		\mu^*f\in H^0(\mu^{-1}(U),K_{\tilde X}\otimes \mu^*\calE(\mu^*h))
		=H^0(U,\mu_*(K_{\tilde X}\otimes \mu^*\calE(\mu^*h))).
	\end{equation*}
	Conversely, if $g\in H^0(\mu^{-1}(U),K_{\tilde X}\otimes \mu^*E(\mu^*h))$,
	we denote by $(\mu^{-1})^*g$ the pullback of the restriction on $\mu^{-1}(U)\backslash\mu^{-1}(A)$ of $g$.
	By a similar argument, $(\mu^{-1})^*g$ can also be extended holomorphically  across the analytic subset $A$, and
	\begin{equation*}
		(\mu^{-1})^*g\in H^0(U,K_X\otimes \calE(h)).
	\end{equation*}
\end{proof}

It is worth mentioning  that the multiplier submodule sheaf itself does not possess the factorial property, motivating us to consider the isomorphism
 \[\pi_*(K_{\mathbb{P}(E^*)}\otimes (L_E)^{m+r}\otimes \pi^*\det E^*)=K_X\otimes S^mE\]
instead of $\pi_*(L_E)^m=S^mE$.

The injectivity theorem for higher direct image sheaves, which has been widely studied in the last decades, is a powerful tool in complex geometry and algebraic geometry (see \cite{Kol86,Tak95,Fuj12,FM21,Mat22}, etc).
A more general statement of injectivity theorem was established by Zhou-Zhu, as follows,

\begin{theorem}[\cite{ZZ}]
\label{ZZ inj}
	Let $X$ be a holomorphically convex K\"ahler manifold,
	$(F,h_F)$ and $(M,h_M)$ be two holomorphic line bundles with singular Hermitian metrics.
	Suppose
	\begin{equation*}
		\cur_{(F,h_F)}\geq0,\ \cur_{(F,h_F)}\geq b\cur_{(M,h_M)}
	\end{equation*}
	in the sense of currents for some positive constant $b$.
	Then for any global holomorphic section $s$ of $M$ satisfying
	\begin{equation*}
		\sup_\Omega |s|_{h_M}<+\infty
	\end{equation*}
	for any $\Omega\Subset X$
	and $q\geq0$, the map
	\begin{equation*}
		H^q(X,K_X\otimes F\otimes\calI(h_F))\xrightarrow{\otimes s}
		H^q(X,K_X\otimes F\otimes M\otimes\calI(h_F\otimes h_M))
	\end{equation*}
	is injective.
	\end{theorem}
As a corollary, one can obtain the torsion-freeness of higher direct image sheaves.
\begin{theorem}[Torsion-freeness, \cite{Mat22}]
\label{torsionfree thm}
	Let $f$ be a proper holomorphic surjective morphism from a K\"ahler manifold $X$ to a complex analytic variety $Y$.
	Suppose $(F,h_F)\ge_\Grif 0$ be a holomorphic line bundle.
	Then, for any $k\geq0$, the sheaf
	\begin{equation*}
		R^kf_*(K_X\otimes F\otimes\calI(h_F))
	\end{equation*}
	is torsion-free.

In particular, $R^kf_*(K_X\otimes F\otimes\calI(h_F))=0$ for $k>\dim X-\dim Y$.
\end{theorem}
\begin{proof}
	
	Suppose there exists $s \in R^kf_*(K_X\otimes F\otimes\calI(h_F))_y$ and $0\neq t\in\calO_{Y,y}$ for a given $y\in Y$, such that $t\cdot s=0$.
	Let us consider a Stein space neighborhood $U$ of $y$ that is sufficiently small to define $s$ and $t$.
	By the assumption that $X$ is a holomorphically convex K\"ahler manifold,
	the preimage $U':=f^{-1}(U)$ is also a holomorphically convex K\"ahler manifold.
	Then $f^* t\in\calO_{X}(U')$.
	
	Applying \[(M,h_M)=(U'\times\CC,1)\]
	in Theorem \ref{ZZ inj},
	then $f^* t$ is a nonzero global section on $M$, satisfying the conditions of Theorem \ref{ZZ inj}.
	Consequently, the map
	\begin{equation*}
		H^k(U',K_X\otimes F\otimes\calI(h_F))\hookrightarrow H^k(U',K_X\otimes F\otimes\calI(h_F)),
	\end{equation*}
	induced by tensoring $f^* t$, is an injection.
	
	Since $t\cdot s=0$, then $f^*t \cdot s=0$, where the later $s$ is regarded as an element of $H^k(U', K_X\otimes F\otimes \calI(h_F))$.
	Therefore, by Theorem \ref{ZZ inj}, we conclude that $s=0$.

This implies that
\begin{equation*}
    R^kf_*(K_X\otimes F\otimes\calI(h_F))
\end{equation*}
is torsion-free, since $y$ is arbitrary.
\end{proof}	

\begin{remark}\label{torsionfree rem}
\begin{enumerate}[(i)]
	\item
	If the torsion-free sheaf $R^kf_* (K_X\otimes F\otimes\calI(h_F))$ is supported on a proper analytic subset, then according to R\"uckert's Nullstellensatz (Theorem \ref{Ruckert Nullstellensatz}), we obtain
	\begin{equation*}
		R^kf_*(K_X\otimes F\otimes\calI(h_F))=0.
	\end{equation*}
	\item
	We can also derive same results to those in Theorem \ref{torsionfree thm} under a milder condition: the existence of an open covering $\{U_\alpha\}$ of $Y$ such that $f^{-1}(U_\alpha)$ is K\"ahler, compared to the K\"ahlerity of $X$.
\end{enumerate}
\end{remark}

It is evident from Lemma \ref{sheaf under modification}, Theorem \ref{torsionfree thm}, and Theorem \ref{Leray thm} that the cohomology groups with multiplier ideal sheaves remain invariant under K\"ahler modifications.
\begin{lemma}\label{lem isom under mod}
  Let $(F,h_F)\ge_\Grif0$ be a holomorphic line bundle over a complex manifold $X$ and $\mu:\tilde{X}\to X$ be a K\"ahler modification. Then for any $q\ge0$,
  \[H^q(X,K_X\otimes F\otimes\calI(h_F))\simeq H^q(\tilde{X},K_{\tilde{X}}\otimes \mu^*F\otimes\calI(\mu^*h_F)).\]
\end{lemma}

\section{Proof of the Main Theorem}
\label{Proof of Main Theorems}
In this section, we present a proof of main Theorem \ref{Isom thm}.
\subsection{{Proof of Theorem \ref{Isom thm}-({\rm i})}}
\label{Proof of Theorem ref isom thm}
		
	In light of Leray's theorem \ref{Leray thm}, our task is to demonstrate the isomorphism of sheaves (Proposition \ref{isom of sheaves})
	and the vanishing of the higher direct image sheaves (Lemma \ref{Direct image vanishing}).
	
	Initially, there is a one-to-one corresponding between $\Gamma(U, S^mE)$ and $ \Gamma(\pi^{-1}(U), L_E^m)$.
	Furthermore, we can extend the isomorphisms twisted with multiplier submodule sheaves.

\begin{proposition}
\label{isom of sheaves}
	Let $(E,h)\ge_\Grif 0$ be a holomorphic vector bundle  on a complex manifold $X$.
	Then we have
   \begin{align*}
   S^m\calE(S^mh)=\pi_*( L_E^{m}\otimes\calI(h_L^{m+r}\otimes\pi^*\det h^*)).
\end{align*}
\end{proposition}

\begin{proof}
	Firstly, we have
	\[S^mE=\pi_*(K_{\PP(E^*)/X}\otimes L_E^{m+r}\otimes\pi^*\det E^*). \]
	Hence, our objective is to show the equivalence of integrability concerning the corresponding sections.

Since $(E,h)$ is Griffiths semi-positive and the question is local, by Proposition \ref{prop lowerbound}, we may assume that $E$ is trivial and there exists a family of Griffith positive smooth Hermitian metrics $\{h_\nu\}$ on $X$ increasingly converging to $h$. Then the metrics $\{h_{\nu,L}\}$ induced by $\{h_\nu\}$ are smooth positive and increasingly converge to  $h_L$. Then by Theorem \ref{thm lsy}, the $L^2$-metric induced by $h_{\nu,L}^{m+r}$ is equal to a constant multiple of $S^mh_\nu\otimes \det h_\nu$.

Consider the setting with $Y := \PP(E^*)$ and
\[(F,h_F):=( L_E^{m+r}\otimes\pi^*\det E^*,h_L^{m+r}\otimes\pi^*\det h^*).\]
For any $x\in X$ with $\det h(x)<+\infty$ and $u\in S^mE_x=H^0(Y_x,K_{Y_x}\otimes F|_{Y_x})$, the induced $L^2$-metric $h_{h_F}$ satisfies
\begin{align}\label{metric formula}
  |u|_{h_{h_F}}^2(x) & :=\int_{Y_x}^{}|u|^2_{h_F} \nonumber\\
   &= \int_{Y_x}^{}|u|^2_{h_L^{m+r}\otimes\pi^*\det h^*} \nonumber\\
   &= \det h^*(x) \int_{Y_x}^{}|u|^2_{h_L^{m+r}}\nonumber\\
   &=\det h^*(x) \lim_{\nu}\int_{Y_x}^{}|u|^2_{h_{\nu,L}^{m+r}}\nonumber\\
   &=C_{m,r}\det h^*(x) \lim_{\nu} |u|^2_{S^mh_\nu}(x)\cdot\det h_\nu(x)\nonumber\\
   &=C_{m,r} |u|^2_{S^mh}(x),
\end{align}
where $C_{m,r}$ is a constant only dependent on $m$ and $r$.

In summary, the $L^2$-metric induced by $h_{L}^{m+r}\otimes\pi^*\det h^*$ is well-defined and equals to a constant multiple of $S^mh$ almost everywhere.
Moreover, for any suitably small open subset $U$ and $u\in H^0(U,S^m\calE(S^mh))$, we conclude from the definition, (\ref{metric formula}) and Fubini's theorem that for each $U'\Subset U$,
\begin{align*}
  +\infty &>\int_{U'}^{}|u|^2_{S^mh}d\lambda(x)  \\
   &=C_{m,r}^{-1}\int_{U'}^{}\int_{Y_x}^{}|u|^2_{h_F}d\lambda(x)  \\
   &=C_{m,r}^{-1}\int_{\pi^{-1}(U')}|u|^2_{h_F}\pi^*d\lambda(x) ,
\end{align*}
where $d\lambda(x)=\sqrt{-1}^{n^2}dx\wedge d\bar{x}$ is a volume form on $U$. This implies that $u\wedge\pi^*dx\in  H^0(\pi^{-1}(U), K_{Y}\otimes F)$ is $L^2$ integrable on $\pi^{-1}(U')$ with respect to $h_F$ and hence
\[u\in H^0(\pi^{-1}(U), K_{Y/X}\otimes F\otimes\calI(h_F)).\]

Conversely, take $u\in H^0(\pi^{-1}(U), K_{Y/X}\otimes F\otimes\calI(h_F))$, then for any $U'\Subset U$,  it follows from the finite covering theorem, (\ref{metric formula}) and Fubini's theorem that $u$ is $L^2$ integrable on $U'$ with respect to $S^m h$. Hence $u\in H^0(U,S^m\calE(S^mh))$.

	Consequently, we establish that
\begin{equation*}
	(K_X\otimes S^m\calE(S^mh))(U)
	\simeq \pi_*(K_{\PP(E^*)}\otimes L_E^{m+r}\otimes\pi^*\det E^*\otimes\calI(h_L^{m+r}\cdot\pi^*\det h^*)) (U).
\end{equation*}
This concludes the proof of Proposition \ref{isom of sheaves}.
\end{proof}

\begin{remark}\label{rmk iso}
  From the proof, we can strengthen  {\textup{Proposition \ref{isom of sheaves}}} to the following result:

  Let $(E,h)\ge_\Grif 0$ be a holomorphic vector bundle and $M$ be a holomorphic line bundle with a singular Hermitian metric $h_M$  on a Hermitian manifold $(X,\omega)$. Assume that \[\cur_{(M,h_M)}\ge-C\omega \] for some constant $C$.
	Then we have
   \begin{align*}
   S^m\calE\otimes\calM(S^mh\otimes h_M)=\pi_*( L_E^{m}\otimes\pi^*M\otimes\calI(h_L^{m+r}\otimes\pi^*\det h^*\otimes\pi^*h_M)),
\end{align*}
where $S^m\calE\otimes \calM(S^mh\otimes h_M)$ is the multiplier submodule sheaf associated to $(S^mE\otimes M,S^mh\otimes h_M)$.
  \end{remark}
	
Recall that Deng-Ning-Wang-Zhou \cite{DNWZ20} introduced a new notion ``optimal $L^2$-estimate condition" and gave a characterization of Nakano positivity for smooth Hermitian metrics on holomorphic vector bundles via this notion.
Subsequently, Inayama \cite{In20} provided a definition of Nakano positivity of singular Hermitian metrics on holomorphic vector bundles based on this characterization.
	\begin{definition}
			[\cite{DNWZ20, In20}]
			Let $(E,h)$ be a Griffiths semi-postive singular Hermitian holomorphic vector bundle.
			 Then $h$ is called Nakano semi-positive, denoted by $(E,h)\ge_\Nak 0$, if
			$h$ is $L^2$ optimal (optimal $L^2$-estimate condition):

   for any Stein coordinate $U$ such that $E|_U$ is trivial, any K\"ahler form $\omega_U$ on $U$, any $\psi \in {\rm Spsh}(U)\cap C^\infty(U)$ and any $\overline{\partial}$-closed $f\in  L^2_{(n,1)} (U,\omega_U,E|_U,he^{-\psi})$, there exists  $u\in  L^2_{(n,0)} (U,\omega_U,E|_U,he^{-\psi})$ such that $\overline{\partial}u = f$ and
			\begin{equation}\label{def Hormander estimate}
			\int_{U}^{} |u|^2_{\omega_U,h} e^{-\psi} dV_{\omega_U}
			\le \int_{U}^{} \langle B^{-1}_{\psi} f,f \rangle_{\omega_U,h} e^{-\psi} dV_{\omega_U},
			\end{equation}
			provided that the right hand side is finite, where $B_{\psi}=[\sqrt{-1}\partial\overline{\partial}\psi \otimes {\rm Id}_E, \Lambda_{\omega_U}]$.
		\end{definition}	

Essentially following the idea of \cite[Theorem 4.3]{DNWZ20} together with  (\ref{metric formula}),  it is deduced that strongly Nakano positivity implies Nakano positivity.
\begin{proposition}
		\label{thm strong implies NP}
		Let $X$ be a complex manifold and $(E,h)\geq_\SNak0$.
		Then $h$ is a Nakano semi-positive singular Hermitian metric.
	\end{proposition}
\begin{proof}
	Let us denote the holomorphic line bundle \[(F, h_F):=(L_E^{r+1}\otimes \pi^*\det E^*, h_L^{r+1}\otimes \pi^*\det h^*).\]
	The condition $(E,h)\geq_\SNak0$ implies $(F, h_F)\ge_\Grif0$.
	By formula (\ref{formula K/X}), we have  $E=\pi_*(K_{\PP(E^*)/X}\otimes F)$.
	
	It suffices to show that $(E, h)$ is $L^2$ optimal.
	Consider any Stein coordinate $(U, z)$ such that $E|_U$ is trivial and any K\"ahler metric $\omega_U$ on $U$.
	Let $f$ be a $\overline{\partial}$-closed smooth $(n, 1)$-form with values in $E$,
	and $\psi$ be any smooth strictly plurisubharmonic function on $U$.
	We express
 \[f (z) = dz \wedge ( f_1(z)d\bar z_1 + \cdots + f_n(z)d\bar z_n),\]
 with $f_i(z) \in E_z = H^0(\PP(E^*)_z, K_{\PP(E^*)_z} \otimes F).$
	One can	identify $f$ as a smooth  $(n+r-1, 1)$-form $\tilde f(z, w) := dz\wedge( f_1(z, w)d\bar z_1+\cdots+ f_n(z, w)d\bar z_n)$
	on $\pi^{-1}(U)$, where $f_i(z, w)$ represents holomorphic sections of $K_{\PP(E^*)}\otimes F|_{\pi^{-1}(z)}$.
	
	We observe the following:
\begin{itemize}
	\item[(a)]
	$\dpa_w f_i(z, w) = 0$ for any fixed $z \in U$, as $f_i(z, w)$ are holomorphic sections of  $K_{\PP(E^*)_z}  \otimes F|_{\pi^{-1}(z)}$.
	\item[(b)]
	$ \dpa_z f = 0$, since $f$ is a $\dpa$-closed form.
\end{itemize}	
	It follows that $\tilde f$ is a smooth $\dpa $-closed $ (n + r-1, 1)$-form on $U\times \PP^{r-1}$ with values in $F$.

Since $f$ is $E$-valued $(n,1)$-form, $\langle B^{-1}_\psi  f,  f\rangle_{\omega_{U},  h}dV_{\omega_{U}}$ is independent of the choice of $\omega_U$ and thus
               \[\langle B^{-1}_\psi  f,  f\rangle_{\omega_{U},  h}dV_{\omega_{U}}=
              \sum_{j,k=1}^{n} \psi^{jk}\langle f_j,f_k\rangle_{h}\sqrt{-1}^{n^2}dz\wedge d\bar z,\]
              where $(\psi^{jk}) = (\frac{\pa^2\psi}{\pa z_j\pa\bar z_k})^{-1}$.

And  since each $f_j|_{\PP(E^*)_z}$ is holomorphic  $F|_{\PP(E^*)_z}$-valued $(r-1,0)$-form, then \[\int_{\PP(E^*)_z}\langle f_j,f_k\rangle_{\omega_{\pi^{-1}(U)}|_{\PP(E^*)_z},h_F}dV_{\omega_{\pi^{-1}(U)}|_{\PP(E^*)_z}}\] does not depend on the choice of $\omega_{\pi^{-1}(U)}$. Hence for any fixed $z\in U$ with $\det h(z)<+\infty$, we have
\begin{align*}
   &\int_{\PP(E^*)_z}\langle B^{-1}_{\pi^*\psi} \tilde f,  \tilde f\rangle_{\omega_{\pi^{-1}(U)}|_{\PP(E^*)_z},  h_F}dV_{\omega_{\pi^{-1}(U)}|_{\PP(E^*)_z}}     \\
   = & \sum_{j,k=1}^{n}\psi^{jk}(z)\left(\int_{\PP(E^*)_z}\langle f_j,f_k\rangle_{\omega_{\pi^{-1}(U)}|_{\PP(E^*)_z},h_F}dV_{\omega_{\pi^{-1}(U)}|_{\PP(E^*)_z}}\right)\sqrt{-1}^{n^2}dz\wedge d\bar z \\
  = &\sum_{j,k=1}^{n}\psi^{jk}(z)\langle f_j,f_k\rangle_{h}\sqrt{-1}^{n^2}dz\wedge d\bar z,
\end{align*}
where the last equality is due to the formula (\ref{metric formula}).

Then by the Fubini-Tonelli theorem, we obtain that
\begin{align*}
            \int_{\pi^{-1}(U)}^{}\langle B^{-1}_{\pi^*\psi}\tilde f,\tilde f\rangle_{\omega_{\pi^{-1}(U)},  h_F}e^{-\pi^*{\psi}}dV_{\omega_{\pi^{-1}(U)}}  =  \int_{U}^{}\langle B^{-1}_\psi  f,  f\rangle_{\omega_{U},  h}e^{-\psi}dV_{\omega_{U}}<+\infty.
\end{align*}

	Notice that $h_F$ is semi-positive on $F$, then there exists an $(n+r-1, 0)$-form $\tilde u$ with $L^2$ coefficients such that  $\dpa \tilde u=\tilde f$ and
\begin{align*}
&\int_{\pi^{-1}(U)} |\tilde u|^2_{\omega_{\pi^{-1}(U)},h_F} e^{-\pi^*\psi}dV_{\omega_{\pi^{-1}(U)}}\\
\le & \int_{\pi^{-1}(U)} \langle B_{\pi^*\psi}^{-1} \tilde f, \tilde f \rangle_{\omega_{\pi^{-1}(U)},h_F} e^{-\pi^*\psi} dV_{\omega_{\pi^{-1}(U)}}\\
=& \int_{U}^{}\langle B^{-1}_\psi  f,  f\rangle_{\omega_{U},  h}e^{-\psi}dV_{\omega_{U}}.
\end{align*}
	
	Additionally, due to the weak regularity of $\dpa$ on $(n+r-1,0)$-forms, we can take $\tilde u$ to be smooth.
	It's observed that $\dpa\tilde u|_{ P(E^*)_z}=0$ for any fixed $z\in U$, due to $\dpa\tilde u=\tilde f$ and $\dpa_w f_i(z, w) = 0$.
	This means that $\tilde u_z:=\tilde u(z,\cdot)\in E_z$.
Consequently, we may view $\tilde u$ as a section $u$ of $E$.
Then $\dpa u=f$ and
\begin{align*}
	\int_{U} | u|^2_{\omega_U,h} e^{-\psi} dV_{\omega_U}
	&=\int_{\pi^{-1}(U)} |\tilde u|^2_{\omega_{\pi^{-1}(U)},h_F} e^{-\pi^*\psi}dV_{\omega_{\pi^{-1}(U)}}\\
 &\le\int_{U}^{}\langle B^{-1}_\psi  f,  f\rangle_{\omega_{U},  h}e^{-\psi}dV_{\omega_{U}}.
\end{align*}

\end{proof}
\begin{remark}
    Similarly, we can show that $(E,h)\ge_\SNak0$ implies that
    $(S^mE,S^mh)\ge_\Nak0$ for any $m\ge1$.
\end{remark}

	We now establish the vanishing of higher direct image sheaves, thereby concluding the proof of Theorem \ref{Isom thm}-({\rm i}).

\begin{lemma}
\label{Direct image vanishing}
	Assume $X$ is a complex manifold.
	Let $(E,h)\ge_\SNak 0$ be a holomorphic vector bundle of rank $r$ on $X$.
	Then for any $k, m>0$, we have
	\begin{equation*}
		R^k\pi_*\left( L_E^{m}\otimes\calI(h_L^{m+r}\otimes\pi^*\det h^*) \right)=0.
	\end{equation*}
\end{lemma}

\begin{proof}	
	To begin, we rephrase the expression:
	\[L_E^{m}\otimes\calI(h_L^{m+r}\otimes\pi^*\det h^*)=K_{\mathbb{P}(E^*)/X}\otimes L_E^{r+m}\otimes\pi^*\det E^*\otimes\calI(h_L^{r+m}\otimes\pi^*\det h^*).\]
	Since $h_L^{m+r}\otimes\pi^*\det h^*$ is  Griffiths semi-positive, then the sheaf $ \calI(h_L^{m+r}\otimes\pi^*\det h^*)$ is coherent.
	This implies that the set
	\[A:=\supp\calO_{\PP(E^*)}\big/ \calI(h_L^{m+r}\otimes\pi^*\det h^*)\] is an analytic subset of $\PP(E^*)$.
	It follows from  Theorem \ref{Remmert's proper mapping theorem} that $\pi(A)$ is an analytic subset of $X$.
	
	Consider a Stein open set $U\subset X$ such that $E|_U$ is trivial.
	For any $x\in\{\det h<+\infty\}$,
	the restricted metric $h_L|_{\pi^{-1}(x)}$ is smooth Griffiths positive. and so is $(h_L^{m+r}\otimes\pi^*\det h^*)|_{\pi^{-1}(x)}$.
	Since the fiber $\PP^{r-1}$ is quasi-Stein,
	then by Ohsawa-Takegoshi $L^2$-extension, we know that such fiber $\pi^{-1}(x)$ does not intersect with $A$.
	Thus, $\pi(A)\subsetneq X$.
	
	Consider the projection map
 \begin{equation*}
	    \pi:(U\setminus \pi(A))\times \PP^{r-1}\to U\setminus \pi(A).
     \end{equation*}
	Since  $ \calI(h_L^{m+r}\otimes\pi^*\det h^*)$  is trivial over $(U\setminus \pi(A))\times \PP^{r-1}$, then
	by Grauert's upper semicontinuity theorem (see Theorem \ref{Grauert's upper semicontinuity}), we deduce that
	\begin{align*}
     x\mapsto {\rm dim} H^k(\pi^{-1}(x),  (K_{\mathbb{P}(E^*)/X}\otimes L_E^{r+m}\otimes\pi^*\det E^*)|_{\pi^{-1}(x)})
   \end{align*}
   is a locally constant function on a non-empty Zariski open set $U'\subset U\setminus \pi(A)$.
   Furthermore, for any $x\in U'$, the following isomorphism holds:
   \begin{align*}
   &\	\frac{(R^k\pi_* (K_{\mathbb{P}(E^*)/X}\otimes L_E^{r+m}\otimes\pi^*\det E^*))_x}{\mathfrak{m}_x(R^k\pi_* (K_{\mathbb{P}(E^*)/X}\otimes L_E^{r+m}\otimes\pi^*\det E^*))_x}
   	\\\simeq \ &\ H^k(\pi^{-1}(x),  (K_{\mathbb{P}(E^*)/X}\otimes L_E^{r+m}\otimes\pi^*\det E^*)|_{\pi^{-1}(x)}).
   \end{align*}

    Additionally, we observe the following:
    \begin{enumerate}[(\rm i)]
        \item
    $\pi^{-1}(x)=\PP^{r-1}$ is compact K\"ahler,
    \item $ (K_{\mathbb{P}(E^*)/X}\otimes L_E^{r+m}\otimes\pi^*\det E^*)|_{\pi^{-1}(x)}=K_{\pi^{-1}(x)}\otimes (L_E^{r+m}\otimes\pi^*\det E^*)|_{\pi^{-1}(x)},$
    \item for $x\in \{\det h<+\infty\}\cap U'$, $(h_L^{r+m}\otimes\pi^*\det h^*)|_{\pi^{-1}(x)}$ is a smooth Griffiths positive metric on $(L_E^{r+m}\otimes\pi^*\det E^*)|_{\pi^{-1}(x)}$.
    \end{enumerate}

   Consequently, by the Kodaira vanishing theorem, it follows that
   \begin{equation*}
       H^k(\pi^{-1}(x),  (K_{\mathbb{P}(E^*)/X}\otimes L_E^{r+m}\otimes\pi^*\det E^*)|_{\pi^{-1}(x)})=0,
       \end{equation*}
   for any $x\in \{\det h<+\infty\}\cap U'$ and any $k>0$.

   Hence by the Nakayama Lemma, for $x\in \{\det h<+\infty\}\cap U'$  and $k>0$, we have
	\begin{equation}
		(R^k\pi_*  (K_{\mathbb{P}(E^*)/X}\otimes L_E^{r+m}\otimes\pi^*\det E^*))_x=0.
	\end{equation}

Therefore, for $x\in \{\det h<+\infty\}\cap U'$  and $k>0$, we obtain
\begin{equation}
		(R^k\pi_*( L_E^{m}\otimes\calI(h_L^{m+r}\otimes\pi^*\det h^*)))_x=0.
	\end{equation}
	Since $R^k\pi_*( L_E^{m}\otimes\calI(h_L^{m+r}\otimes\pi^*\det h^*))$ is coherent, then it is supported on a proper analytic subset of $U$.

Additionally, combining with
	 Theorem \ref{torsionfree thm} and Remark \ref{torsionfree rem}-$({\rm ii})$, it can be concluded that
	\begin{equation*}
	R^k\pi_*( L_E^{m}\otimes\calI(h_L^{m+r}\otimes\pi^*\det h^*))
	\end{equation*}
	is torsion-free.
	Thus, according to Remark \ref{torsionfree rem}-$({\rm i})$, we infer that
	\begin{equation*}
	     R^k\pi_*( L_E^{m}\otimes\calI(h_L^{m+r}\otimes\pi^*\det h^*))=0.
      \end{equation*}
\end{proof}

\begin{remark}\label{rmk van}
 The proof allows us to strengthen {\textup{Lemma \ref{Direct image vanishing}}} to the following assertion:

  Let $(E,h)$ be a holomorphic vector bundle and $(M,h_M)$ be a holomorphic line bundle on a  Hermitian manifold $(X,\omega)$. Suppose that \begin{equation*}
      \cur_{(L_E^{m+r},h_L^{m+r})}-\pi^*\cur_{(\det E,\det h)}+\pi^*\cur_{(M,h_M)}\ge-C\pi^*\omega
      \end{equation*}
      for some constant $C$.
	Then, for any $k>0$, we obtain
	\begin{align*}
   R^k\pi_*( L_E^{m}\otimes\pi^*M\otimes\calI(h_L^{m+r}\otimes\pi^*\det h^*\otimes\pi^*h_M))=0.
	\end{align*}

\end{remark}

Moreover, a more general statement than Theorem \ref{Isom thm}  can be obtained via Remark \ref{rmk iso}, Theorem \ref{Leray thm} and Remark \ref{rmk van}.

 \begin{theorem}\label{thm twist iso}
   Let $(E,h)\ge_\Grif 0$ and $W$ be  holomorphic vector bundles and $(M,h_M)$ be a holomorphic line bundle on a Hermitian manifold $(X,\omega)$. Assume that
   \begin{equation*}
       \cur_{(M,h_M)}\ge -C\omega
       \end{equation*}
   and
   \begin{equation*}
   \cur_{(L_E^{m+r},h_L^{m+r})}-\pi^*\cur_{(\det E,\det h)}+\pi^*\cur_{(M,h_M)}\ge-C\pi^*\omega
   \end{equation*}
   for some positive constant $C$.
	Then for any $p,q\ge0,m>0$, we have
   \begin{align*}
   &H^q(X,W\otimes S^m\calE\otimes\calM(S^mh\otimes h_M))\\\simeq&H^q(\PP(E^*), \pi^*W\otimes L_E^{m}\otimes\pi^*M\otimes\calI(h_L^{m+r}\otimes\pi^*\det h^*\otimes\pi^*h_M)).
\end{align*}
 \end{theorem}

\subsection{Proof of Theorem \ref{Isom thm}-({\rm ii})}
	To establish the second result of Theorem $\ref{Isom thm}$, we consider $W=\Omega_X^p$ and $m=1$ in Theorem \ref{Isom thm}-({\rm i}), which yields an isomorphism of cohomologies:
\begin{align}\label{formula isom of E and F^p}
	H^q(X, \Omega_X^p\otimes \calE(h))
	\simeq H^q(\PP(E^*), \pi^*\Omega_X^p \otimes L_E\otimes\calI(h_L^{1+r}\otimes\pi^*\det h^*)).
\end{align}
	Hence, the rest is to construct an isomorphism between the cohomologies of $\pi^*\Omega_X^p \otimes L_E\otimes\calI(h_L^{1+r}\otimes\pi^*\det h^*)$ and $\Omega_{\PP(E^*)}^p\otimes  L_E\otimes\calI(h_L^{1+r}\otimes\pi^*\det h^*)$.

	For convenience, we denote these sheaves by
\begin{align*}
	&\calF^p:=\pi^*\Omega_X^p\otimes  L_E\otimes\calI(h_L^{1+r}\otimes\pi^*\det h^*)\\
	&\calG^p:=\Omega_{\PP(E^*)}^p\otimes  L_E\otimes\calI(h_L^{1+r}\otimes\pi^*\det h^*).
\end{align*}

	Our goal now is to find an open covering such that $\calF^p$ and $\calG^p$ have the same sections on these open sets, and their higher cohomologies both vanish.
	
\begin{lemma}
\label{relation of H0}
	Let $U\subset X$ be a Stein open set such that $E$ is locally trivial on $U$.
 Then the inclusion $\calF^p\subset \calG^p$ induces an isomorphism
		  \begin{equation}
		  \label{local isom induced by interm sheaf}
		  	H^0(\pi^{-1}(U),\calF^p)
		  	\simeq
		  	H^0(\pi^{-1}(U), \calG^p);
		  \end{equation}
\end{lemma}
\begin{proof}
	Since the multiplier ideal sheaves on both sides of (\ref{local isom induced by interm sheaf}) are identical, it suffices to show the isomorphism of bundles on both sides of (\ref{local isom induced by interm sheaf}).
	
	For any $x\in U$, we have $L_E|_{\pi^{-1}(x)}=\calO_{\PP^{r-1}}(1)$ and $(\pi^*\Lambda^pT^*X)|_{\pi^{-1}(x)}$ forms a product bundle.
	From the decomposition
	\begin{equation*}
		(\Lambda^pT^*\PP(E^*))|_{\pi^{-1}(x)}
		=\bigoplus_{s+t=p} (\pi^*\Lambda^sT^*_xX\otimes\Lambda^tT^*\PP(E^*_x)),
	\end{equation*}
	we deduce that
	\begin{align*}
		H^0(\pi^{-1}(x),\Omega_{\PP(E^*)}^p\otimes L_E)
		=&\bigoplus_{s+t=p}(\Lambda^sT^*_xX)\otimes H^0(\pi^{-1}(x), \Omega_{\PP(E^*_x)}^t\otimes\calO_{\PP^{r-1}}(1) )    \\
		=&(\Lambda^pT^*_xX)\otimes H^0(\pi^{-1}(x),L_E|_{\pi^{-1}(x)}),
	\end{align*}
	where the second equality is owing to a consequence of Bott's vanishing theorem (see Theorem \ref{Bott vanishing}):
	\begin{equation*}
		H^0(\PP^{r-1},\Omega^t_{\PP^{r-1}}\otimes\calO_{\PP^{r-1}}(1))=0,
		\text{ for any }t>0.
	\end{equation*}

\end{proof}
\begin{lemma}
\label{Vanishing of Hq}
	As above notation, then for any $p\geq0$, $q>0$,
	\begin{enumerate}[$(\rm i)$]
		\item \label{vanishing of calF p}
		$H^q(\pi^{-1}(U),\calF^p)=0$ ,
		\item \label{vanishing of Omega PP p}
		$H^q(\pi^{-1}(U),\calG^p)=0$.
	\end{enumerate}
\end{lemma}
\begin{proof}
	(i) As in the proof of Lemma \ref{Direct image vanishing}, we obtain that
	 \begin{align*}
   	\frac{(R^q\pi_*\calF^p )_x}{\mathfrak{m}_x(R^q\pi_*\calF^p)_x}
   \simeq&   H^q(\pi^{-1}(x), \calF^p|_{\pi^{-1}(x)})\\
   \simeq&   H^q(\PP^{r-1}, \calO_{\PP^{r-1}}(1))=0,
   \end{align*}
   for $x\in U'\cap\{\det h<+\infty\}$ where $U'\subseteq U$ is a Zariski open set in $U$.
 Consequently, we obtain
 \begin{equation*}
 R^q\pi_*\calF^p =0.
  \end{equation*}
	Hence, $H^q(\pi^{-1}(U),\calF^p)=0$.
	
    (ii) Similarly,  for $x\in U'\cap\{\det h<+\infty\}$, where $U'\subseteq U$ is a Zariski open set in $U$,
 we derive that
    \begin{align*}
   	\frac{(R^q\pi_*\calG^p)_x}{\mathfrak{m}_x(R^q\pi_*\calG^p)_x}
   	\simeq& H^q(\PP^{r-1}, (\oplus_{0\leq t\leq p}\Omega_{\PP^{r-1}}^t)\otimes \calO_{\PP^{r-1}}(1))=0,
   \end{align*}
   by Bott's vanishing theorem on $\PP^{r-1}$ (see Theorem \ref{Bott vanishing}).
\end{proof}

\vskip0.7cm

	We now complete the proof of Theorem \ref{Isom thm}-(${\rm ii}$).
	Take a Stein open covering $\{U_j\}$ of $X$ such that $E|_{U_j}$ is trivial.
	Denoted by $V_j:=\pi^{-1}(U_j)$, then $\{V_j\}$ forms an open covering of $\PP(E^*)$.
	By Lemma \ref{Vanishing of Hq}, the covering $\{V_j\}$ is a Leray covering for sheaves $\calF^p$ and $\calG^p$
	such that the sheaf cohomologies $H^*(\PP(E^*),\calF^p)$ and $H^*(\PP(E^*),\calG^p)$ can be computed using the open covering $\{V_j\}$.
	Together with  Lemma \ref{relation of H0}, we arrive at
	\begin{equation}
	\label{isom of Hq for calFp and ideal}
		H^q(\PP(E^*),\calF^p)
		\simeq
		H^q(\PP(E^*),\calG^p).
	\end{equation}
	By combining (\ref{formula isom of E and F^p}) and (\ref{isom of Hq for calFp and ideal}), we derive the result.

\section{Applications}
\label{Applications}
 In this section, we give some applications of Theorem \ref{Isom thm} and Theorem \ref{Isom thm'}.

 \subsection{Coherence and Strong openness}

Suppose $(E,h)\ge_\Grif 0$ holds, Inayama in \cite{In22} raised the question whether the submodule sheaf $\calE(h)$ is coherent.  By leveraging Proposition \ref{isom of sheaves} and Grauert's direct image theorem, it suffices to show the coherence of the ideal sheaf $\mathcal{I}(h_L^{1+r}\otimes\pi^*\det h^*)$.
Moreover, building upon the idea of Inayama \cite{In22} and Zou \cite{Zou22}, one can prove that $\calI(h_L^{1+r}\otimes\pi^*\det h^*)$ is coherent when $\det h$ has a discrete unbounded locus or analytic singularities.
Here the unbounded locus  of $\det h$ is the complement of the maximal subset such that $\det h$ is locally bounded. Therefore, we obtain the following result.

\begin{proposition}
\label{Coherence}
    Let $(E,h)\ge_\Grif 0$.
    \begin{enumerate}[$(\rm i)$]
        \item  If $\calI(h_L^{1+r}\otimes\pi^*\det h^*)$ is coherent, then $\calE(h)$ is coherent.
        \item  If $\det h$ has a discrete unbounded locus or analytic singularities, then  $\calI(h_L^{1+r}\otimes\pi^*\det h^*)$ is coherent.
    \end{enumerate}
\end{proposition}

The subsequent result is the strong openness of the multiplier submodule sheaf associated to strongly Nakano semi-positive Hermitian vector bundles.

\begin{proposition}
	Let $(E,h)\ge_\SNak 0$ be a holomorphic vector bundle on a complex manifold $X$.
	Then for any $m\geq0$,
		$S^m\calE(S^mh)$
	 satisfies the strong openness property, which means that one has
	\begin{align*}
		\bigcup_j S^m\calE(S^mh_j)
		=S^m\calE(S^mh)
	\end{align*}
	for a decreasing sequence $\{h_j\}$ of Griffiths semi-positive metrics converging to $h$.
\end{proposition}
\begin{proof}
	As the result is local, we may assume that $E$ is a trivial vector bundle on the unit polydisk $\Delta^n$ and $h\ge {\rm Id}_E$.
	
	Initially, we consider the the special case that $h_\varepsilon=h(\det h )^\varepsilon$, regarded as a new metric on $E$ that converges decreasingly to $h$ as $\varepsilon$ tends to $0$.
	Since
	\begin{equation*}
	    h_{\varepsilon,L}^{m+r}\otimes \pi^*\det h^*_\varepsilon
	=h_L^{m+r}\otimes \pi^*\det h^* \cdot (\det h)^{m\varepsilon}
 \end{equation*}
	is also Griffiths semi-positive and decreasingly converges to $h_L^{m+r}\otimes \pi^*\det h^* $,
it follows from Theorem \ref{Isom thm} and the strong openness conjecture proved by Guan-Zhou \cite{GZ15} that
    \begin{equation*}
        \bigcup_{\varepsilon>0} S^m\calE(S^mh_\varepsilon)
	=S^m\calE(S^mh).
  \end{equation*}	

	For the general case, one can check that $S^mh\cdot (\det h)^{-m}$ is locally bounded from above and  decreases with respect to $h$.
    In other words, if $h'\ge h$, then
    \begin{equation*}
         S^mh'\cdot (\det h')^{-m}\le S^mh\cdot (\det h)^{-m}.
         \end{equation*}
	Take a holomorphic section $F$ of $S^mE$ such that  $|F|^2_{S^mh}$ is integrable on $\Delta^n$ and a compact subset $K\subset \Delta^n$.
	Denote $-\log\det h_j$ and $-\log\det h$ by $\varphi_j$ and $\varphi$ respectively, which are plurisubharmonic functions. This leads to
\begin{align*}
	\int_K |F|^2_{S^mh_j}\le \int_K |F|^2_{S^mh} e^{m\varphi-m\varphi_j}.
\end{align*}

	We will show that $|F|^2_{S^mh}\in L^{1+\varepsilon}(K)$ for some  sufficiently small $\varepsilon>0$
	and for any fixed $q>1$, $e^{m\varphi-m\varphi_j}\in L^q(K)$ for $j>>1$.
	These facts imply $|F|^2_{S^mh_j}\in L^1(K)$ for $j>>1$ by H\"older's inequality.

	On the one hand, We may assume $|F|^2_{S^mh}\cdot (\det h)^{-m}<C$ on $K$ due to the local boundedness of $S^mh\cdot (\det h)^{-m}$,
	then
	\begin{equation*}
	    (|F|^2_{S^mh})^{1+\varepsilon}\le C^\varepsilon |F|^2_{S^mh}(\det h)^{m\varepsilon},
     \end{equation*}
	where the right hand side is integrable on $K$ for  some $\varepsilon>0$ owing to the strong openness  of the special case.
	
	On the other hand, it follows from Demailly's approximation theorem \cite[Theorem 14.2]{D12b} that there are a finite number of holomorphic functions $f_\alpha, \alpha=1,\ldots, N,$ on $\Delta^n$ with $\int_{\Delta^n} |f_\alpha|^2e^{-qm\varphi}=1$ and
	\begin{equation*}
	     qm\varphi\le \log(\sum_\alpha  |f_\alpha|^2)+O(1) \quad\text{on}~ K.
      \end{equation*}
	Then
	$e^{qm\varphi-qm\varphi_j}\le C\sum_{\alpha} |f_\alpha|^2 e^{-qm\varphi_j}$ for some constant $C_K$ on $K$.
	As $qm\varphi_j\le qm\varphi$ increasingly converges to $qm\varphi$,  it follows that	
	 \begin{equation*}
	     \int_K e^{qm\varphi-qm\varphi_j}\le C_K\sum_{\alpha=1}^{N}\int_K  |f_\alpha|^2 e^{-qm\varphi_j}<+\infty
      \end{equation*}
	for $j$ large enough due to the strong openness conjecture proved by Guan-Zhou \cite{GZ15} .
	
\end{proof}	
	\begin{remark}
	  In fact, since the strongly Nakano positivity implies Nakano positivity by Proposition \ref{thm strong implies NP}, we can obtain  strong openness  of multiplier submodule sheaves directly from \cite[Theorem 1.3]{LYZ21}.
	\end{remark}

\subsection{Direct sum and a counterexample}
\label{Direct sum and isomorphism}
In this subsection, we consider the direct sum of holomorphic vector bundles with direct sum metrics.

\begin{corollary}\label{cor direct sum}
	Let $(E_i,h_i)\ge_\SNak 0$ be holomorphic vector bundles of rank $r_i$ on a complex manifold $X$, for $i=1,\ldots,\ell$.
	If $W$ is a holomorphic vector bundle over $X$, then for $q\ge 0$, we have
\begin{align*}
		&H^q(X, W\otimes (\calE_1\oplus\cdots\oplus\calE_\ell)(h_1\oplus\cdots\oplus h_\ell)) \\
		\simeq&
		H^q(\P(E_1^*),\pi_1^*W\otimes L_{E_1}\otimes\calI(h_{L_{E_1}}^{1+r_1}\otimes\pi_1^*\det h_1^*) ) \\
		&\qquad\oplus\cdots\oplus
		H^q(\P(E_\ell^*),\pi_\ell^*W\otimes L_{E_\ell}\otimes\calI(h_{L_{E_\ell}}^{1+r_\ell}\otimes\pi_\ell^*\det h_\ell^*) )
\end{align*}
and
\begin{align*}
&H^q(X, \Omega_X^p\otimes (\calE_1\oplus\cdots\oplus\calE_\ell)(h_1\oplus\cdots\oplus h_\ell))  \\
		\simeq&
		H^q(\P(E_1^*),\Omega_{\PP(E_1^*)}^p\otimes L_{E_1}\otimes\calI(h_{L_{E_1}}^{1+r_1}\otimes\pi_1^*\det h_1^*) ) \\
	  &\qquad\ \oplus\cdots\oplus
	  H^q(\P(E_\ell^*),\Omega_{\PP(E_\ell^*)}^p\otimes L_{E_\ell}\otimes\calI(h_{L_{E_\ell}}^{1+r_\ell}\otimes\pi_\ell^*\det h_\ell^*) ),
\end{align*}
where $\pi_i:\PP(E_i^*)\rightarrow X$ for $i=1,\ldots,\ell$.
\end{corollary}
\begin{proof}
	By the definition of multiplier submodule sheaves, we have
	\begin{align*}
	(\calE_1\oplus\cdots\oplus\calE_\ell)(h_1\oplus\cdots\oplus h_\ell)
	=\calE_1(h_1)\oplus\cdots\oplus\calE_\ell(h_\ell).
	\end{align*}
	Since cohomologies commute with direct sum of sheaves, we obtain
	\begin{align*}
		&H^q(X, W\otimes (\calE_1\oplus\cdots\oplus\calE_\ell)(h_1\oplus\cdots\oplus h_\ell))    \\
	\simeq& H^q(X,W\otimes\calE_1(h_1))\oplus\cdots\oplus H^q(X,W\otimes\calE_\ell(h_\ell))
	\end{align*}
	and
	\begin{align*}
	 &H^q(X, \Omega_X^p\otimes (\calE_1\oplus\cdots\oplus\calE_\ell)(h_1\oplus\cdots\oplus h_\ell))    \\
	\simeq& H^q(X,\Omega_X^p\otimes\calE_1(h_1))\oplus\cdots\oplus H^q(X,\Omega_X^p\otimes\calE_\ell(h_\ell)).
	\end{align*}
	By Theorem \ref{Isom thm}, we complete the proof.
	\end{proof}

In \cite{Wu20}, Wu proposed a problem whether a Nakano positive vector bundle always admits a strongly Nakano positive Hermitian metric. Here we construct a counterexample to disprove his problem.
\begin{example}\label{nak not snak}
	Let $L$ be a ample line bundle and $\underline{\CC}$ be the trivial line bundle over a compact Riemann surface $\Sigma$.
	We consider the vector bundle $E:=L\oplus\underline{\CC}$.
	On the one hand, let $h$ be a positive metric on $L$.
	Then the metric $h\oplus \Id$
	on $E$ is Nakano semi-positive.
	
	On the other hand,  $Y:=\PP(E^*)=\PP(L^*\oplus\underline{\CC})$ is the so-called \textit{Hirzebruch-like surface} with $\ell:=\deg(L^*)<0$. More discussions on the Hirzebruch-like surfaces can be found in \cite{Gau}.	Subsequently, we do not distinguish divisors and line bundles for simplicity and abuse the notations of divisors and line bundles when no confusion can arise.
	
	The Hirzebruch-like surface $Y=\PP(L^*\oplus\underline{\CC})$ can be viewed as the compactification of $L^*$ by adding a point at infinity to each fiber $L_x^*$.
	The union of these points is called the \textit{infinity section}, denoted by $\Sigma_\infty$.
	Whereas the zero section of $L^*$ is regarded as a section of $\PP(L^*\oplus\underline\CC)$ via the inclusion $L^*\subset\PP(L^*\oplus\underline\CC)$, denoted by $\Sigma_0$.
	
	It is well-known that
	\begin{equation*}
		\calO_{\PP(E^*)}(1)
		=\calO_Y(\Sigma_\infty).
	\end{equation*}
	Indeed, as stated in \cite[Proposition 7.12]{Har77}, the projection of $L^*\oplus\underline\CC$ to the trivial line bundle $\underline\CC$ determines a section of $\calO_{\PP(E^*)}(1)$, whose zero divisor is $\Sigma_\infty$.
	Similarly, the projection of $L^*\oplus\underline\CC$ to $L^*$ determines a section of $\calO_{\PP(E^*)}(1)\otimes\pi^*L^*$, and its zero divisor is $\Sigma_0$.
	This leads to the relation
	\begin{equation*}
		\Sigma_0=\Sigma_\infty-\pi^*L.
	\end{equation*}
	Moreover, we also have
	\begin{equation*}
		\Sigma_\infty\cdot\Sigma_\infty=-\ell,\quad
		\Sigma_0\cdot\Sigma_0=\ell,\quad
		\Sigma_0\cdot\Sigma_\infty=0.
	\end{equation*}
	
	Noticing that $\det(E)=\det(L)=L$,
	  we get that
	\begin{align*}
		(r+1)\calO_{\PP(E^*)}(1)-\pi^*\det E
		=3\Sigma_\infty-\pi^*L
		=2\Sigma_\infty+\Sigma_0.
	\end{align*}
	Hence we obtain that
	\begin{align*}
		((r+1)\calO_{\PP(E^*)}(1)-\pi^*\det E)\cdot\Sigma_0
		=\Sigma_0\cdot\Sigma_0=\ell<0.
	\end{align*}
	Therefore, the line bundle $(r+1)\calO_{\PP(E^*)}(1)-\pi^*\det E
$  can not be semi-positive.
In other words, the vector bundle $E=L\oplus\underline\CC$ is not strongly Nakano semi-positive.

By twisting the above vector bundle $E$ with a small ample line bundle, we also provide an example, which is Nakano positive but not strongly Nakano positive.

Concretely, consider the vector bundle $E:=L^N\oplus\underline\CC$ for $N>1$ and
\begin{equation*}
	F:=L^{N+1}\oplus L
	=L\otimes E.
\end{equation*}
 Then we have
\begin{equation*}
\PP(F^*)\simeq\PP(E^*)\quad\text{ and }\quad
	\calO_{\PP(F^*)}(1)
	=\calO_{\PP(E^*)}(1)+\pi^*L.
\end{equation*}
One computes
\begin{align*}
	3\calO_{\PP(F^*)}(1)-\pi^*\det(F)
	=&3\calO_{\PP(E^*)}(1)+3\pi^*L-\pi^*\det(L^{N+1}\oplus L) \\
	=&3\calO_{\PP(E^*)}(1)-(N-1)\pi^*L   \\
	=&3\Sigma_\infty-\frac{N-1}{N}(\Sigma_\infty-\Sigma_0)   \\
	=&(2+\frac{1}{N})\Sigma_\infty+ (1-\frac{1}{N})\Sigma_0.
\end{align*}
Then we obtain that
\begin{align*}
	(3\calO_{\PP(F^*)}(1)-\pi^*\det(F))\cdot\Sigma_0
	=(1-\frac{1}{N})N\deg(L^*)<0.
\end{align*}
Thus, the vector bundle $F$ over $\Sigma$ is not strongly Nakano positive.

\end{example}

\subsection{Injectivity theorem}

In this subsection, we generalize Zhou-Zhu's injectivity theorem for pseudo-effective line bundles (see Theorem \ref{ZZ inj})  to strongly Nakano semi-positive holomorphic vector bundles.
\begin{theorem}
\label{inj for vb}
	Let $(X,\omega)$ be a holomorphically convex K\"ahler manifold and $(E,h)\geq_\SNak0$.
	Assume that
	\begin{equation*}
		(E,h)\geq_\SNak b\sqrt{-1} \Theta_{(M,h_M)}
	\end{equation*}
	for some $0<b<+\infty$,
	and
	\begin{equation*}
		\cur_{(M,h_M)}\ge-C\omega
	\end{equation*}
for some constant $C$.
		Then for any non-zero $s\in H^0(X,M)$ satisfying
	\begin{equation*}
		\sup_\Omega |s|_{h_M}<+\infty
	\end{equation*}
	for every $\Omega\Subset X$, we obtain that the map
	\begin{equation*}
		H^q(X,K_X\otimes S^m\calE(S^mh))\xrightarrow{\otimes s}
		H^q(X,K_X\otimes S^m\calE\otimes \calM(S^mh\otimes h_M))
	\end{equation*}
	is injective for any $q\ge 0$.

\end{theorem}

\begin{proof}
By setting $W=K_X$ in Theorem \ref{Isom thm}-({\rm i}) and observing that
\begin{equation*}
K_{\PP(E^*)/X}:=K_{\PP(E^*)}\otimes\pi^*K_X^{-1}=L_E^{-r}\otimes\pi^*\det E,
\end{equation*}
we derive the relationship
	\begin{equation*}
		H^q(X,K_X\otimes S^m\calE(S^mh))
		\simeq H^q(\PP(E^*),K_{\PP(E^*)}\otimes F\otimes\calI(h_F)),
	\end{equation*}
	where $F=L_E^{m+r}\otimes\pi^*\det E^*$ and $h_F=h_L^{m+r}\otimes\pi^*\det h^*$.

	Similarly, due to  $\cur_{(M,h_M)}\ge-C\omega$,  we obtain by Theorem \ref{thm twist iso} that
	\begin{equation*}
		H^q(\PP(E^*),K_{\PP(E^*)}\otimes F\otimes \pi^*M\otimes\calI(h_F\otimes \pi^*h_M))
        \simeq
        H^q(X,K_X\otimes S^m\calE\otimes \calM(S^mh\otimes h_M)).
	\end{equation*}

	In addition, since $(E,h)\geq_\SNak0$ and $(E,h)\geq_\SNak\sqrt{-1} b\Theta_{(M,h_M)}$, it follows from Theorem \ref{ZZ inj}  that $\pi^*s$ induces a injective morphism
	\begin{equation*}
		H^q(\PP(E^*),K_{\PP(E^*)}\otimes F\otimes\calI(h_F))\xrightarrow{\otimes \pi^*s}
		H^q(\PP(E^*),K_{\PP(E^*)}\otimes F\otimes \pi^*M\otimes(h_F\otimes \pi^*h_M))
	\end{equation*}
	 for any $q\geq0$.
	Hence we conclude that  the map
	\begin{equation*}
		H^q(X,K_X\otimes S^m\calE(S^mh))\xrightarrow{\otimes s}
		H^q(X,K_X\otimes S^m\calE\otimes \calM(S^mh\otimes h_M))
	\end{equation*}
	is injective for any $q\ge 0$.
\end{proof}

	Following a similar process to that in Theorem \ref{torsionfree thm}, we also deduce that
\begin{corollary}[Torsion-freeness]
\label{Cor torsion-free}
	Let $X$ be a K\"ahler manifold
	and $f$ be a proper holomorphic surjective morphism from $X$ to a complex analytic variety $Y$.
	Suppose that $(E,h)\geq_\SNak0$ is a holomorphic vector bundle over $X$.
	Then
	 $R^kf_*(K_X\otimes S^m\calE(S^mh))$ is torsion-free for any $k\geq0$.
	In particular, if $k>\dim X-\dim Y$, then
	\begin{equation*}
 R^kf_*(K_X\otimes S^m\calE(S^mh))=0.
 \end{equation*}
\end{corollary}

Let $M$ be a holomorphic line bundle over a Hermitian manifold $(X,\omega)$. Suppose that $h_M$ is a singular metric on $M$ satisfying that
\begin{equation*}
		\cur_{(M,h_M)}\ge-C\omega
	\end{equation*}
for some constant $C$. We denote
\begin{equation*}
	H^0_{{\rm bdd}, h_M}(X,M)
	:=\{s\in H^0(X,M)~|~\sup_\Omega|s|_{h_M}<+\infty {\textup{\ for\ any }} \Omega\Subset X\}.
\end{equation*}
Assume that $\{h_k\}_{k=1}^{+\infty}$ is a sequence of singular metrics on holomorphic line bundle $M$.
The \textit{generalized Kodaira dimension} is defined as
\begin{equation*}
\kappa_{\bdd}(M, \{h_k\}_{k=1}^{+\infty}):=
	\begin{cases}
		-\infty, \qquad\qquad\qquad\qquad \text{if } H^0_{\bdd,(h_k)^k}(X,M^k)=0 \text{ for }k>>0, \\
		\sup\left\{m\in\ZZ\left|\limsup\limits_{k\to\infty}\frac{\dim H^0_{\bdd,(h_k)^k}(X,M^k)}{k^m}>0\right.\right\}, \quad \text{otherwise}.
	\end{cases}
\end{equation*}

	We can derive the following vanishing result as a corollary of Theorem~ \ref{inj for vb}.
\begin{theorem}\label{thm Kbdd vanish}
	Let $(X,\omega)$ be a projective manifold and $(E,h)\geq_\SNak0$ be a holomorphic vector bundle over $X$.
	Consider $M$ as a holomorphic line bundle equipped with a sequence  of singular metrics $\{h_k\}_{k=1}^{+\infty}$ over $X$.
	Assume that the following two inequalities hold on $X$ in the sense of currents:
	\begin{enumerate}[$(\rm i)$]
		\item $(E,h)\ge_\SNak \varepsilon_k\cur_{(M,h_k)}$ for any $k\ge1$ and some sequence of positive numbers
  $\{\varepsilon_k\}_{k=1}^{+\infty}$,
		\item $k\cur_{(M,h_k)}\ge-C\omega$ for any $k\in\NN$ and some constant $C>0$.
	\end{enumerate}
	Then we have
	\begin{equation*}
		H^q(X,K_X\otimes S^m\calE(S^mh))=0
	\end{equation*}
	for $q>n-\kappa_\bdd(M,\{h_k\}_{k=1}^{+\infty})$.
\end{theorem}
\begin{proof}
	Denote $\kappa_\bdd(M,\{h_k\}_{k=1}^{+\infty})$ by $\kappa$.
   There is nothing if $\kappa\le0$.

	We only need to consider the case $\kappa>0$. By contradiction, suppose that there exists a non-zero cohomology class
	\begin{equation*}
		\xi\in H^q(X,K_X\otimes S^m\calE(S^mh))
	\end{equation*}
	for some  $q>n-\kappa$.

	For $k>0$,  the following map $T_\xi$ induced by the tensor product with $\xi$
	\begin{equation*}
		H^0_{\bdd,(h_k)^k}(X,M^k)\longrightarrow
		H^q(X,K_X\otimes S^m\calE\otimes \calM^k(S^mh\otimes(h_k)^k))
	\end{equation*}
	 is a linear map.
	
	By assumptions that $(E,h)\ge_\SNak 0$, $(E,h)\ge_\SNak \varepsilon_k\cur_{(M,h_k)}$ and $k\cur_{(M,h_k)}\ge-C\omega$,  applying Theorem \ref{inj for vb} for the above linear map $T_\xi$,
	we obtain that $T_\xi$ is injective.
	Hence, we derive
	\begin{equation*}
		\dim H^0_{\bdd,(h_k)^k}(X,M^k)
		\leq
		\dim H^q(X,K_X\otimes S^m\calE\otimes \calM^k(S^mh\otimes(h_k)^k)).
	\end{equation*}
	According to the following Lemma \ref{asymp of Hq}, one has
		\begin{equation*}
			\dim H^q(X,K_X\otimes S^m\calE\otimes \calM^k(S^mh\otimes(h_k)^k))
			=O(k^{n-q}) \qquad\text{as }k\rightarrow+\infty.
		\end{equation*}
	Then we have
	\begin{equation*}
		\dim H^0_{\bdd,(h_k)^k}(X,M^k)
		=O(k^{n-q}) \qquad\text{as }k\rightarrow+\infty,
	\end{equation*}
	which implies that $\kappa\le n-q$, contradicting the assumption $q>n-\kappa$.
\end{proof}
\begin{lemma}
\label{asymp of Hq}
	With the above notation, we conclude that
	\begin{equation*}
			\dim H^q(X,K_X\otimes S^mE\otimes M^k(S^mh\otimes(h_k)^k))
			=O(k^{n-q}) \qquad\text{as }k\rightarrow+\infty.
		\end{equation*}
\end{lemma}	
\begin{proof}
We prove by induction on the dimension of manifolds.
For simplicity, we denote
\begin{equation*}
	\calI_k:=S^m\calE\otimes \calM^k(S^mh\otimes(h_k)^k).
\end{equation*}

	For a given $k>0$, since $\calI_k$ is a torsion-free coherent sheaf by Corollary \ref{Cor torsion-free}, we can take a very ample smooth divisor $H$ so  that $H$ does not contain any component of subvarieties defined by the associated primes of $\calI_k$. Then
	we obtain a short exact sequence
	\begin{equation*}
		0\rightarrow K_X\otimes \calI_k
		\xrightarrow{\iota} K_X\otimes H\otimes\calI_k
		\rightarrow (K_X\otimes H\otimes\calI_k)|_H
		\rightarrow0.
	\end{equation*}

	The above short exact sequence induces a long exact sequence
	\begin{equation*}
		\cdots\rightarrow H^{q-1}(H,(K_X\otimes H\otimes\calI_k)|_H)
		\rightarrow H^q(X,K_X\otimes\calI_k)
		\rightarrow H^q(X,K_X\otimes H\otimes\calI_k) \rightarrow\cdots.
	\end{equation*}
	Using Theorem \ref{Isom thm}, we deduce the following isomorphism
	\begin{align}\label{isom for calIk}
		&H^q(X,K_X\otimes H\otimes\calI_k) \\
		\simeq
		&H^q(\PP(E^*),K_{\PP(E^*)}\otimes  L_E^{m+r}\otimes\pi^*(\det E^*\otimes H\otimes M^k)\otimes\calI(h_L^{m+r}\otimes\pi^*(\det h^*\otimes(h_k)^k)) )\nonumber.
	\end{align}
	Since $(E,h)\ge_\SNak0$ and $ k\cur_{(M,h_k)}\ge-C\omega$, we can take $H$ very ample such that
	\begin{equation*}
		\ndim( L_E^{m+r}\otimes\pi^*(\det E^*\otimes H\otimes M^k),h_L^{m+r}\otimes\pi^*(\det h^*\otimes h_H\otimes(h_k)^k))
		=n+r-1,
	\end{equation*}
	where $h_H$ is a smooth Griffiths positive metric on $H$ satisfying
\begin{equation*}
    \cur_{(H,h_H)}\ge (C+1)\omega.
    \end{equation*}
	Then  for  $q,k>0$, it follows from Theorem \ref{K-V-N-C} and (\ref{isom for calIk}) that
	\begin{equation*}
		H^q(X,K_X\otimes H\otimes\calI_k) =0.
	\end{equation*}
	Together with the above long exact sequence, we get
	that
\begin{equation*}
    \dim H^0(X,K_X\otimes\calI_k)\leq  O(k^n)
    \end{equation*}
thanks to the Riemann-Roch formula.
Additionly, we also obtain
	\begin{equation}
	\label{induction formu}
	\dim H^q(X,K_X\otimes\calI_k)
		\leq \dim H^{q-1}(H,(K_X\otimes H\otimes\calI_k)|_H)\leq O(k^{n-q})
	\end{equation}
 for $q>0$ by induction and the adjunction formula $(K_X\otimes H)|_H=K_H$.

	When $n=1$, we can take $H=\sum_{j=1}^Np_j$ with $p_j\notin\supp(S^mE\otimes M^k/\calI_k)$ and $N$ independent of $k$.
 Furthermore, we have
	\begin{align*}
		\dim H^1(X,K_X\otimes\calI_k)
		\leq &\dim H^0(H, (K_X\otimes H\otimes\calI_k)|_H)\\
		= &\sum_{j=1}^N \dim H^0(p_j, (K_X\otimes H\otimes\calI_k)|_H)  \\
		=& \sum_{j=1}^N \dim H^0(p_j, (K_X\otimes H\otimes S^mE\otimes M^k)|_H)\\
		=&N {\textup{\ rank}}(S^m E).
	\end{align*}
\end{proof}

  Theorem \ref{thm Kbdd vanish} contains two distinctive cases.

  The first case occurs when  $h_k= h_M, k\in\ZZ_+$ for some fixed singular metric $h_M$ on $M$. Then the curvature assumptions ({\rm i}) and  ({\rm ii}) are equivalent to the single assumption
  \begin{equation*}
      (E,h)\ge_\SNak \varepsilon\cur_{(M,h_M)}\ge0
      \end{equation*}
   for some positive number $\eps$.  If $h_M$ is further a smooth Griffiths positive metric, then
   \begin{equation*}
       \kappa_\bdd(M,\{h_k\}_{k=1}^{+\infty})=n.
       \end{equation*}
   Consequently, Theorem \ref{thm Kbdd vanish} transforms into a Nadel-type vanishing theorem on projective manifolds.

   The second case arises when $M$ is numerically effective which means that there exists a sequence of smooth metrics $\{h_k\}_{k=1}^{+\infty}$ such that
   \begin{equation*}
       \cur_{(M,h_k)}\ge -\frac{1}{k}\omega.
       \end{equation*}
   In this case, the curvature assumptions  ({\rm ii}) is automatically fulfilled, and   $\kappa_\bdd(M,\{h_k\}_{k=1}^{+\infty})$ is equal to the usual Kodaira dimension of $M$.

\subsection{Vanishing theorems}
\label{Nadel-Cao type vanishing theorem}

It is clear that a Kawamata-Viehweg-Nadel-type vanishing theorem for holomorphic vector bundles with strongly Nakano semi-positive metrics can be derived from Theorem \ref{Isom thm} and Theorem \ref{K-V-N-C}.
\begin{corollary}
	\label{Cao type vanishing}
	Let $X$ be a compact K\"ahler manifold and $(E, h)\geq_\SNak 0$ be a holomorphic vector bundle over $X$.
	Then for any
	$m\ge1$, we have
	\begin{equation*}
	H^q(X,K_X\otimes S^m\calE(S^mh))=0
	\end{equation*}
	for  $ q\ge n+r-{\rm nd}(L_E^{m+r}\otimes\pi^*\det E^*, h_L^{m+r}\otimes\pi^*\det h^*)$.
\end{corollary}

	In fact, $(E, h)\geq_\SNak 0$ implies that
 \begin{equation*}
      {\rm nd}(L_E^{m+r}\otimes\pi^*\det E^*, h_L^{m+r}\otimes\pi^*\det h^*)\ge r-1.
      \end{equation*}

	Let us go into details.
	 Take an open set $U\subset X$ such that $E|_U$ is trivial. Through a direct computation from (\ref{induced metric h_L}) on $U\times\PP^{r-1}\subset\PP(E^*)$,
 we obtain
	\begin{equation*}
		\cur_{(L_E,h_L)}\geq  b\omega_\FS,
	\end{equation*}
	where $\omega_\FS$ is regarded as the pull-back of the Fubini-Study metric on $\PP^{r-1}$ and  $b(z)$ is a non-negative continuous function satisfying that $b>0$ outside a pluripolar set in $U$.
	Then there exists a positive constant  $c>0$ and a  subset $B\subseteq U$ of positive measure satisfying
	\begin{equation*}
		\cur_{(L_E,h_L)}\geq c\omega_\FS \text{ on }B\times \PP^{r-1}.
	\end{equation*}
	By Demailly's approximation theorem, there exists a quasi-equisingular approximation  $\{h_\varepsilon\}$  on $L_E$ over $\PP(E^*)$ such that \begin{equation*}
	    \cur_{(L_E,h_\varepsilon)}
		  \ge (b-\frac{1}{2}\eps)   \omega_{\rm FS}.
    \end{equation*}
	In particular, on $B\times \PP^{r-1}$  we have
\begin{equation*}
    	\cur_{(L_E,h_\varepsilon)}\geq \frac{c}{2}\omega_\FS,
\end{equation*}
for sufficiently small $\eps$.
    As $X$ is compact K\"ahler and so is $\PP(E^*)$,
    by the definition of numerical dimension, we derive that
    \begin{enumerate}[(i)]
        \item ${\rm nd}(L_E,h_L)\ge r-1$ if $(E,h)\ge_\Grif 0$,
        \item ${\rm nd}(L_E^{m+r}\otimes\pi^*\det E^*, h_L^{m+r}\otimes\pi^*\det h^*)\ge r-1,$ if $(E,h)\ge_\SNak 0$.
    \end{enumerate}

Moreover, if $h$ is  Griffiths positive, then we can get
	\begin{equation*}
		\cur_{(L_E,h_L)}\geq  a\pi^*\omega+b\omega_\FS,
	\end{equation*}
	where	$a(z)$ is  a positive continuous function
	 in $U$.
	Consequently, there exists a positive constant $c>0$ and a subset $B\subseteq U$ of positive measure satisfying
	\begin{equation*}
		\cur_{(L_E,h_L)}\geq c(\pi^*\omega+\omega_\FS) \text{ on }B\times \PP^{r-1}.
	\end{equation*}
	Subsequently, a similar argument implies that
\begin{equation*}
	{\rm nd}(L_E,h_L)=n+r-1.
\end{equation*}
Furthermore, if $(E,h)>_\SNak 0$, then
\begin{equation*}
 {\rm nd}(L_E^{m+r}\otimes\pi^*\det E^*, h_L^{m+r}\otimes\pi^*\det h^*)=n+r-1.
 \end{equation*}

\begin{corollary}\label{coro vanisging for SNak>0}
 Let  $(E, h)>_\SNak 0$  be a holomorphic vector bundle of rank $r$ over a compact $($possibly non-K\"ahler$)$ manifold $X$ of dimension $n$, then for any $m\ge1$, we have
	\begin{equation*}
		H^q(X,K_X\otimes S^m\calE(S^mh))=0
	\end{equation*}
	for  $ q\ge  1$.
\end{corollary}
\begin{proof}
  As  $(E, h)>_\SNak 0$, we have
   \begin{equation*}
\int_{\PP(E^*)}\left(\cur_{(L_E^{m+r}\otimes\pi^*\det E^*, h_L^{m+r}\otimes\pi^*\det h^*)}\right)^{n+r-1}>0.
\end{equation*}
   Thus,  Lemma \ref{lem kahler mod} implies that $\PP(E^*)$ admits a K\"ahler modification $\mu:\tilde{P}\to\PP(E^*).$
   Then applying Theorem \ref{Isom thm} and Lemma \ref{lem isom under mod}, we obtain  that
   \begin{align*}
      & H^q(X,K_X\otimes S^m\calE(S^mh)) \\
     \simeq &  H^q({\PP(E^*)},K_{\PP(E^*)}\otimes L_E^{m+r}\otimes\pi^*\det E^*\otimes\calI(h_L^{m+r}\otimes\pi^*\det h^*)) \\
      \simeq &  H^q(\tilde{P},K_{\tilde{P}}\otimes \mu^*(L_E^{m+r}\otimes\pi^*\det E^*)\otimes\calI(\mu^*(h_L^{m+r}\otimes\pi^*\det h^*))).
   \end{align*}
   Observing that
   \begin{equation*}
   \int_{\tilde{P}}\left(\cur_{(\mu^*(L_E^{m+r}\otimes\pi^*\det E^*), \mu^*(h_L^{m+r}\otimes\pi^*\det h^*))}\right)^{n+r-1}>0,\end{equation*}
    we obtain by Theorem \ref{K-V-N-C} that
   \begin{equation*}
		H^q(X,K_X\otimes S^m\calE(S^mh))=0,
	\end{equation*}
for any $ q\ge  1$.
\end{proof}
Similarly, one can show that
\begin{corollary}[Griffiths-type vanishing]
 Let  $(E, h)>_\Grif 0$  be a holomorphic vector bundle of rank $r$ over a compact $($possibly non-K\"ahler$)$ manifold $X$ of dimension $n$, then for any $m\ge1$, we have
	\begin{equation*}
		H^q(X,K_X\otimes S^m\calE\otimes\det\calE(S^mh\otimes\det h))=0
	\end{equation*}
	for  $ q\ge  1$.
\end{corollary}

The subsequent result serves as a generalization of Meng-Zhou's vanishing theorem for vector bundles with singular Hermitian metrics.
\begin{corollary}
	\label{Meng-Zhou type vanishing}
	Let $X$ be a holomorphically convex K\"ahler manifold, and let $(E,h)\geq_\Grif0$ be a holomorphic vector bundle on $X$.
	Suppose $h$ is smooth outside an analytic subset $Z$ and satisfies
	\begin{equation}
	\label{pair nonzero}
		\left(\langle \sqrt{-1}\Theta(h^*)a,a\rangle(x)\right)^k\neq0.
	\end{equation}
	for some point $x\in X\backslash Z$ and $a\in E^*_x$.
	Then we have
	\begin{equation*}
		H^q(X,K_X\otimes S^m\calE\otimes\det \calE(S^mh\otimes\det h))=0.
	\end{equation*}
    for any $q\geq n-k+1$.
\end{corollary}
\begin{proof}
	Let $\{e_\alpha\}$ be a normal frame of $E$ at $x\in X\backslash Z$.
The curvature of $h$ at $x$ is written as
\begin{equation*}
	\Theta_{(E,h)}=c_{ij\alpha\beta}dz^i\wedge d\bar z^j\otimes e^*_\alpha\otimes e_\beta ,
\end{equation*}
satisfying $\bar c_{ij\alpha\beta}=c_{ji\beta\alpha}$. Additionally, the curvature of $h^*$ at $x$ is written as
\begin{equation*}
	\Theta_{(E,h^*)}=-c_{ij\beta\alpha}dz^i\wedge d\bar z^j\otimes e_\alpha\otimes e^*_\beta.
\end{equation*}

Write $a\in E^*_x\setminus\{0\}$ as $\sum a_\alpha e^*_\alpha$, then the curvature of $h_L$ at $(x,[a])$ is given by
\begin{flalign}
	\Theta_{(L_E,h_L)}(x,[a])
	=&\sum c_{ij\beta\alpha }\frac{a_\alpha\bar{a}_\beta}{|a|^2} dz^i\wedge d\bar z^j +\sum dw^\lambda\wedge d\bar w^\lambda \nonumber \\
 =&-\frac{\langle \Theta(h^*)a,a\rangle}{|a|^2}
	+\sum dw^\lambda\wedge d\bar w^\lambda.
	\label{curvature of LE normal1}
\end{flalign}

Using (\ref{pair nonzero}) and (\ref{curvature of LE normal1}), we obtain that
\begin{equation*}
	\left(\sqrt{-1}\Theta_{(L_E,h_L)}(x,[a])\right)^{k+r-1}\neq0.
\end{equation*}
Therefore, we conclude by Theorem \ref{Isom thm} and Theorem \ref{MZ} that
\begin{align*}
		&H^q(X,K_X\otimes \calE\otimes\det \calE(h\otimes\det h)) \\
		=& H^q(\PP(E^*),K_{\PP(E^*)}\otimes L_E^{r+1}\otimes\calI(h_L^{r+1}))
		=0.
	\end{align*}
    for any $q\geq n+r-1-(k+r-1)+1=n-k+1$.

For $m\geq2$, the proof follows a similar argument.
\end{proof}

\subsection{Singular holomorphic Morse inequalities for vector bundles}
\label{Singular holomorphic Morse inequalities for vector bundles}

Firstly, let us recall Bonavero's singular holomorphic Morse inequalities.
\begin{theorem}[\cite{Bon98}]
\label{Bonavero's singular H-M-Ineq}
	Let $Y$ be a compact complex manifold of dimension $n$ and let $L$ be a holomorphic line bundle on $Y$ equipped with a Hermitian metric $h_L$ with analytic singularities.
	Take $V$ as a holomorphic vector bundle on $Y$.
	Then for $0\leq q\leq n$, we have
	\begin{align}\label{weak sing Morse ineq}
		h^q(Y,V\otimes L^m(h_L^m)) \leq\rk(V)\dfrac{m^n}{n!}\int_{Y(q)}(-1)^q(\cur_{(L,h_L)})^n +o(m^n),
	\end{align}
	and
	\begin{align}\label{strong sing Morse ineq}
		\sum_{j=0}^q(-1)^{q-j} h^j(Y,V\otimes L^m(h_L^m)) \leq\rk(V)\dfrac{m^n}{n!}\int_{Y(\leq q)}(-1)^q(\cur_{(L,h_L)})^n +o(m^n).
	\end{align}
	where
	\begin{align*}
		Y(q):=\{y\in Y| &\cur_{(L,h_L)} \text{ has }q \text{ negative eigenvalues and } n-q \text{ positive eigenvalues} \},
	\end{align*}
	and
	\begin{equation*}
		Y(\leq q):=\bigcup_{j=0}^qY(q).
	\end{equation*}
\end{theorem}

According to Theorem \ref{thm twist iso}, it is immediate to obtain a singular holomorphic Morse inequalities for vector bundle, as follows,
\begin{theorem}
\label{Singular Morse for v.b.}
	Let $X$ be a compact complex manifold of dimension $n$ and $(E,h)$ be a holomorphic vector bundle endowed with a singular Hermitian metric. Assume that there is a smooth function $\phi$ such that $(E,he^{-\phi})\ge_\Grif 0$ and the induced metric $h_L$ on $L_E$ over $\PP(E^*)$ has analytic singularities.
	Take $V$ as a holomorphic vector bundle on $X$.
	Then for $0\leq q\leq n$, we have
	\begin{align*}
		&h^q(X, V\otimes S^m\calE\otimes\det \calE(S^mh\otimes\det h))  \\
	\leq&\rk(V)\frac{m^{n+r-1}}{(n+r-1)!} \int_{\PP(E^*)(q)} (-1)^q(\cur_{(L_E,h_L)})^{n+r-1} +o(m^{n+r-1}),
	\end{align*}
	and
	\begin{align}
	&\sum_{j=0}^q(-1)^{q-j} h^j(X, V\otimes S^m\calE\otimes\det \calE(S^mh\otimes\det h)) \nonumber \\
		\leq&\rk(V)\frac{m^{n+r-1}}{(n+r-1)!} \int_{\PP(E^*)(\leq q)} (-1)^q(\cur_{(L_E,h_L)})^{n+r-1} +o(m^{n+r-1}).
	\label{Strong singular Hol Morse ineq. of V.B.}
\end{align}
	where similarly
	\begin{align*}
	\PP(E^*)(q):=\{y\in \PP(E^*)| &\text{ the curvature }\cur_{(L_E,h_L)} \text{ has }q \text{ negative eigenvalues } \\
	&\ \ \text{and } n+r-1-q \text{ positive eigenvalues} \},
	\end{align*}
	and
	\begin{equation*}
	\PP(E^*)(\leq q):=\bigcup_{j=0}^q\PP(E^*)(q).
	\end{equation*}
\end{theorem}

\begin{proof}Since  $h_L$ with analytic singularities,
	it follows from Theorem \ref{thm twist iso} and Theorem \ref{Bonavero's singular H-M-Ineq} that  for any $0\leq q\leq n$
	\begin{align*}
	&h^q(X, V\otimes S^m\calE\otimes\det \calE(S^mh\otimes\det h))          \\
	=& h^q(\PP(E^*),K_{\PP(E^*)/X}\otimes\pi^*V\otimes L_E^{m+r}\otimes\calI(h_L^{m+r}))  \\
	\leq&\rk(V)\frac{(m+r)^{n+r-1}}{(n+r-1)!} \int_{\PP(E^*)(q)} (-1)^q(\cur_{(L_E,h_L)})^{n+r-1} +o(m^{n+r-1}),
	\end{align*}
	and
	\begin{align*}
	&\sum_{j=0}^q(-1)^{q-j} h^j(X, V\otimes S^m\calE\otimes\det \calE(S^mh\otimes\det h)) \nonumber \\
	=&\sum_{j=0}^q(-1)^{q-j} h^j(\PP(E^*),K_{\PP(E^*)/X}\otimes\pi^*V\otimes L_E^{m+r}\otimes\calI(h_L^{m+r})) \nonumber \\
	\leq&\rk(V)\frac{(m+r)^{n+r-1}}{(n+r-1)!} \int_{\PP(E^*)(\leq q)} (-1)^q(\cur_{(L_E,h_L)})^{n+r-1} +o(m^{n+r-1}).
	\end{align*}
\end{proof}

\begin{example}\label{eg 1}
    Let $E$ be a trivial holomorphic vector bundle of rank $r$ over a domain $U\subset\CC^n$. Let $\{F_k\}_{k\in\ZZ_+}\subset \mathcal{O}(U)$ be a sequence of holomorphic functions satisfying the following conditions:
    \begin{enumerate}[(i)]
    	\item
     For any $K\Subset U$, there exists $C_K>0$ and $N_K\in\ZZ_+$ such that  on $K$
    \begin{equation*} \sum_{k=1}^{+\infty} |F_{k}|^2\le C_K\sum_{k=1}^{N_K}|F_k|^2;\end{equation*}
    \item  $\dim \mathrm{Span}_{k\in\ZZ_+}\{F_k(z)\}=r$ outside a proper analytic subset $A\subset U$.
    \end{enumerate}

    For any fixed holomorphic frame $e=\{e_j\}^r_{j=1}$, one can define a singular metric $h^*=_{e^*} H$ on $E^*$ using the following matrix representation:
    \begin{equation*}
    H=\sum_{k} (F_{k,i}\overline{F}_{k,j})_{r\times r}
    =\begin{pmatrix}
    \sum_{k} F_{k,1}\overline{F}_{k,1}
    & \cdots
    &\sum_{k} F_{k,1}\overline{F}_{k,r}\\
    \vdots &\ddots & \vdots\\
    \sum_{k} F_{k,r}\overline{F}_{k,1}
    & \cdots
    &\sum_{k} F_{k,r}\overline{F}_{k,r}\\
    \end{pmatrix}.
    \end{equation*}
   Then, for  any local holomorphic section $G=\sum_{j=1}^rG_je^*_j$ of $E^*$,
   we obtain that
   \begin{equation*}
   |G|^2_{h^*}=\sum_{i,j}(\sum_{k} F_{k,i}\overline{F}_{k,j})G_i\overline{G}_j=\sum_{k}|\sum_i F_{k,i}G_i|^2
   \end{equation*}
     is a plurisubharmonic function.
    Therefore, $h^*$ is Griffiths semi-negative.
    Consequently, $(E,h)\geq_\Grif0$.
    Additionally, according to $(\ref{induced metric h_L})$, the induced metric $h_L$ on $L_E$ is represented by
    \begin{equation*}
    \varphi_h(z,w)=\log\sum_{k}|\sum_i F_{k,i}w_i|^2.
    \end{equation*}
    Therefore, condition (1) implies that $h_L$ has analytic singularities.
\end{example}

\begin{example}\label{eg 2}
      Let $E$ be a holomorphic vector bundle of rank $r$ over an n-dimensional compact K\"ahler manifold $X$.
      Assuming $\{F_k\}$ forms a basis of $H^0(X,E)$ such that outside a proper analytic subset $A\subset X$:
\begin{equation*}
\dim \mathrm{Span}_{k}\{F_k(z)\} = r.
\end{equation*}

    With a fixed holomorphic frame $e=\{e_j\}^r_{j=1}$, expressing $F_k$ as $F_k=\sum F_{k,j}e_j$, one can define a singular metric $h^*=_{e^*} H$ on $E^*$ represented by the following matrix:
    \begin{equation*}
    H=\sum_{k} (F_{k,i}\overline{F}_{k,j})_{r\times r}
    =\begin{pmatrix}
    \sum_{k} F_{k,1}\overline{F}_{k,1}
    & \cdots
    &\sum_{k} F_{k,1}\overline{F}_{k,r}\\
    \vdots &\ddots & \vdots\\
    \sum_{k} F_{k,r}\overline{F}_{k,1}
    & \cdots
    &\sum_{k} F_{k,r}\overline{F}_{k,r}\\
    \end{pmatrix}.
    \end{equation*}
   Therefore, $(E,h)\geq_\Grif0$, and the induced metric $h_L$ on $L_E$ has analytic singularities.
   \end{example}

In addition, Theorem \ref{Bonavero's singular H-M-Ineq} yields an explicit formula about the asymptotic growth of global sections of pseudo-effective line bundle, which is equipped with a singular metric with analytic singularities.

Let $(M,h_M)\geq_\Grif0$ be a holomorphic line bundle and $V$ be a holomorphic vector bundle on a compact K\"ahler manifold $Y$.
Assume that $h_M$ has analytic singularities. On the one hand,
taking $q=0$ in (\ref{weak sing Morse ineq}), we have
\begin{align}
h^0(Y,V\otimes M^k\otimes\calI(h_M^k))
\leq\rk(V) \frac{k^n}{n!}\int_Y(\cur_{(M,h_M)})^n+o(k^n).
\end{align}
On the other hand,
taking $q=1$ in (\ref{strong sing Morse ineq}), we obtain that
\begin{align*}
h^0(Y,V\otimes M^k\otimes\calI(h_M^k))
\geq\rk(V)\frac{k^{n}}{n!} \int_{Y} (\cur_{(M,h_M)})^{n} +o(k^{n}).
\end{align*}
Thus, we conclude that
\begin{equation}
\label{asymt formula for analytic sing}
\lim_{k\rightarrow\infty} \frac{1}{k^n}h^0(Y,V\otimes M^k\otimes\calI(h_M^k))
=\frac{\rk(V)}{n!}\int_Y(\cur_{(M,h_M)})^n.
\end{equation}

Fix a smooth metric $h_0$ on $M$ and denote $\theta:=\cur_{(M,h_0)}$. We denote by $\psh(Y,\theta)$ the space of quasi-plurisubharmonic function $\varphi$ on $Y$ such that $\theta+\ddbar \varphi\geq0$ in the sense of currents.
For any $\varphi\in\psh(Y,\theta)$, we have $(M,h_0e^{-\varphi})\geq_\Grif0$.
The so-called $\calI$-\textit{model envelope} is defined as
\begin{equation}
\label{calI-envelope}
P[\varphi]_\calI(z)
:=\sup\{\psi(z) ~|~ \psi \in\psh(Y,\theta), \psi\leq0,
\calI(k\psi)\subseteq\calI(k\varphi), \forall\ k\in\NN\}
\end{equation}
for any $\varphi\in\psh(Y,\theta)$.
It is observed that $P[\varphi]_\calI \in\psh(Y,\theta)$.
The $\calI$-model envelope was initially explored in \cite{KS20} and deeply studied in \cite{DX21}, \cite{DX22}.
A quasi-plurisubharmonic function $\varphi\in\psh(Y,\theta)$ is called $\calI$-\textit{model} if $P[\varphi]_\calI=\varphi$.

Very recently, the authors in \cite{DX21} extended the formula (\ref{asymt formula for analytic sing}) to the case of $\calI$-model quasi-plurisubharmonic function.

\begin{theorem}[\cite{DX21}]
	\label{DX asymptotic}
	Let $(M,h_0e^{-\varphi})\geq_\Grif0$ be a Hermitian line bundle
	and $T$ be a holomorphic vector bundle of rank $s$ on a compact K\"ahler manifold $Y$.
	Then
	\begin{equation}
	\lim_{k\rightarrow\infty}\frac{1}{k^n} h^0(Y,T\otimes M^k\otimes\calI(k\varphi))
	=\frac{s}{n!}\int_Y \left(\cur_{(M,h_0e^{-P[\varphi]_\calI})}\right)^{n},
	\end{equation}
	where $(\cur_{(M,h_0e^{-P[\varphi]_\calI})})^{n}$ is the Monge-Amp\`ere measure of $P[\varphi]_\calI$ in the sense of $($\textup{\ref{Cao's definition}}$)$.
\end{theorem}
\begin{remark}
\label{monotonicity of MA measure}
	It follows the definition (\ref{calI-envelope}) of $\calI$-envelope that $u\leq  P[u]_{\calI}+C$ for some constant $C$.
	According to the monotonicity of Monge-Amp\`ere masses \cite[Theorem 1.1]{WN19}, one obtains that
	\begin{equation*}
		\int_Y \left(\cur_{(M,h_0e^{-P[\varphi]_\calI})}\right)^{n}
	\geq\int_Y\left(\cur_{(M,h_0e^{-\varphi})}\right)^n .
	\end{equation*}
\end{remark}
Take $(M,h_M)=(\det E, \det h)$ and $W=\calO_X$ in Theorem \ref{thm twist iso} and Theorem \ref{DX asymptotic}. If $\PP(E^*)$ is compact K\"ahler, then we obtain that
\begin{align*}
&\lim_{m\rightarrow\infty}\frac{1}{m^{n+r-1}}
h^0(X,S^m\calE\otimes\det \calE(S^mh\otimes\det h))     \\
=&\lim_{m\rightarrow\infty}\frac{1}{m^{n+r-1}}
h^0(\PP(E^*),K_{\PP(E^*)/X}\otimes L_E^{m+r}(h_L^{m+r}) )        \\
=&\frac{1}{(n+r-1)!}\int_{\PP(E^*)} \left(\cur_{(L_E,P[h_L]_\calI)}\right)^{n+r-1}.
\end{align*}
Thus, we get an asymptotic formula for holomorphic vector bundles,
\begin{theorem}
\label{Darvas-Xia type}
	Let $X$ be a compact K\"ahler  manifold and $(E,h)\geq_\Grif0$.
	Then we have
\begin{align*}
	&\lim_{m\rightarrow\infty}\frac{1}{m^{n+r-1}}
	h^0(X,S^m\calE\otimes\det \calE(S^mh\otimes\det h))  \\
	 =&\frac{1}{(n+r-1)!}\int_{\PP(E^*)} \left(\cur_{(L_E,P[h_L]_\calI)}\right)^{n+r-1}.
\end{align*}
\end{theorem}

Lastly, We conclude this subsection by highlighting an
asymptotic inequality for Griffiths positive vector bundles which should be analogous to that of big line bundles.

\begin{theorem}
\label{thm nonvan}
	Let  $(E, h)\ge_\Grif 0$  be a holomorphic vector bundle of rank $r$ over a compact $($possibly non-K\"ahler$)$ manifold $X$ of dimension $n$. Then we have	\begin{align}\label{bbbb}
	\lim_{m\rightarrow\infty}\frac{1}{m^{n+r-1}}
	h^0(X,S^m\calE(S^mh))
	\geq\frac{1}{(n+r-1)!}\int_{\PP(E^*)} \left(\cur_{(L_E,h_L)}\right)^{n+r-1}.
\end{align}
\end{theorem}
\begin{proof}
	Without loss of generality, we may assume that
	\begin{equation*}\int_{\PP(E^*)}\left(\cur_{(L_E,h_L)}\right)^{n+r-1}>0.
 \end{equation*}
	Thus, Lemma \ref{lem kahler mod} implies that $\PP(E^*)$ admits a K\"ahler modification $\mu:\tilde{P}\to\PP(E^*).$
	Then using Lemma \ref{lem isom under mod}, we establish the relation
	\begin{align}\label{for **}
	&H^q(\PP(E^*),K_{\PP(E^*)}\otimes \pi^*\det E^*\otimes L_E^{m+r}(h_L^{m+r})) \nonumber \\
	\simeq&   H^q(\tilde{P},K_{\tilde{P}}\otimes\mu^*\pi^*\det E^* \otimes \mu^*L_E^{m+r}\otimes\calI(\mu^*h_L^{m+r})).
	\end{align}
	Additionally, it's apparent that
	\begin{equation*}
 \int_{\tilde{P}}\left(\cur_{(\mu^*L_E^{m+r}, \mu^*h_L^{m+r})}\right)^{n+r-1}>0.
 \end{equation*}
	Using Proposition \ref{isom of sheaves}, Theorem \ref{Darvas-Xia type} together with the isomorphism (\ref{for **}), we obtain that
\begin{align*}
	&\lim_{m\rightarrow\infty}\frac{1}{m^{n+r-1}}
	h^0(X, S^m\calE(S^mh))     \\
  =&\lim_{m\rightarrow\infty}\frac{1}{m^{n+r-1}}
  h^0(\PP(E^*),K_{\PP(E^*)/X}\otimes L_E^{m+r}\otimes\pi^*\det E^*(h_L^{m+r}\otimes\pi^*\det h^*) )   \\
  \geq&\lim_{m\rightarrow\infty}\frac{1}{m^{n+r-1}}
  h^0(\PP(E^*),K_{\PP(E^*)/X}\otimes\pi^*\det E^*\otimes L_E^{m+r}(h_L^{m+r}) )
  \\
  =&\lim_{m\rightarrow\infty}\frac{1}{m^{n+r-1}}  h^0(\tilde{P},K_{\tilde{P}/X}\otimes\mu^*\pi^*\det E^* \otimes \mu^*L_E^{m+r}\otimes\calI(\mu^*h_L^{m+r})) \\
  \geq&\frac{1}{(n+r-1)!}\int_{\tilde{P}}\left(\cur_{(\mu^*L_E, \mu^*h_L)}\right)^{n+r-1}\\
  =&\frac{1}{(n+r-1)!}\int_{{\PP(E^*)}}\left(\cur_{(L_E, h_L)}\right)^{n+r-1},
  \end{align*}
where  the last inequality is due to Theorem \ref{DX asymptotic} and Remark \ref{monotonicity of MA measure}.
\end{proof}
\begin{remark}
If $h$ is Griffiths semi-positive on $X$ and merely positive on an open neighborhood, then the integral on the right-hand side of $(\ref{bbbb})$ is positive, which means that $H^0(X,S^m\calE(S^mh))$ is not zero for $m>>1$.
\end{remark}

\begin{corollary}
  Let $(E, h)>_\Grif 0$  be a holomorphic vector bundle of rank $r$ over a compact Hermitian manifold $(X,\omega)$ of dimension $n$.
  Consider $\{x_1,\ldots,x_N\}$ be a set of finite points and  $s_j\ge 0$ for $j=1,\ldots N$. Then for $m>>1$, there exists a nonzero global section $u$ of $S^m E$ vanishing at  $x_j$  with order at least $s_j$, $j=1,\ldots,N$.
  \end{corollary}
  \begin{proof}
    By partition of unit, there exists a function $\theta$ on $X$ such that $\theta$ is smooth outside  $\{x_1,\ldots,x_N\}$ and equals to $2(n+s_j)\log |z-x_j|$ near $x_j$ for each $j$. Since $h$ is Griffiths positive, by definition one can show that  there exists $m_0$ such that $e^{-\frac{1}{m_0}\theta} h$ is also Griffiths positive. Then for $m>m_0$ large enough, it follows from Theorem \ref{thm nonvan} that there is a  nonzero global section $u$ of $S^m E$ with
    \begin{equation*}
    \int_X |u|^2_{S^m h}e^{-\frac{m}{m_0}\theta}dV_\omega<+\infty.
    \end{equation*}
    As $h$ is locally bounded below by a continuous metric and $e^{-\theta}$ is equals to $\frac{1}{|z-x_j|^{2(n+s_j)}}$ near $x_j$ for each $j$, we obtain that $u$ vanishes at  $x_j$  with order at least $s_j$, $j=1,\ldots,N$, due to $m>m_0$.
  \end{proof}

\small
\noindent {\sc Department of Mathematics,
University of Maryland,
4176 Campus Dr,
College Park, MD 20742,USA; Tsinghua University,
	Department of Mathematics and Yau Mathematical Sciences Center,
	No 30, Shuangqing Road, Beijing 100084,
	China.}

{\tt yxliu238@umd.edu}\vspace{0.1in}

\noindent {\sc Beijing Institute of Mathematical Sciences and Applications, Beijing 101408, China; Department of Mathematics and Yau Mathematical Sciences Center, Tsinghua University, Beijing 100084, China.}

{\tt liuzhuo@amss.ac.cn}\vspace{0.1in}

\noindent {\sc School of Mathematical Sciences, Peking University;
Institute of Mathematics, Academy of Mathematics and Systems Science, Chinese Academy of Sciences, Beijing 100190, P. R. China
}

{\tt yanghui@amss.ac.cn}\vspace{0.1in}

\noindent {\sc Institute of Mathematics, Academy of Mathematics and Systems Science, Beijing 100190, P. R. China}

{\tt xyzhou@math.ac.cn}\vspace{0.1in}

\end{document}